\documentclass{article}
\usepackage{amsfonts}
\usepackage{amssymb}
\usepackage{amsmath}
\usepackage{fancyhdr}
\usepackage{bm}
\usepackage{color}
\usepackage{graphicx}
\usepackage{psfrag}
\usepackage{indentfirst}
\usepackage{latexsym}
\usepackage{mathrsfs}
\usepackage{array}
\usepackage{verbatim}
\usepackage[tight, footnotesize]{subfigure}
\usepackage{rotating}
\usepackage{multirow}
\usepackage{pgfplots}
\usepackage{tikz}
\usepackage{graphics} 
\usepackage{epsfig} 
\usepackage{hyperref}
\allowdisplaybreaks

\newtheorem{Theorem}{\color{black} Theorem}[section]
\newtheorem{assumption}{\color{black} Assumption}[section]
\newtheorem{lemma}[Theorem]{\color{black} Lemma}

\newtheorem{definition}{\color{black} Definition}[section]
\newtheorem{Remark}[Theorem]{\color{black} Remark}

\numberwithin{equation}{section}
\begin{document}
	
	\title{Amplitude equations for SPDEs with quadratic nonlinearities forced by  additive and multiplicative noise\thanks{This work is supported by National Natural Science Foundation of China (No.11771177) and Jilin Scientific and Technological Development Program (No.20190201132JC and No.20200201264JC).}}
	
	
	\author{Shiduo Qu \\
		School of Mathematics, Jilin University 
		Changchun 130012
		P. R. China\\
		Emails: qusdjlu@hotmail.com\\
		Wenlei Li\\
		School of Mathematics, Jilin University
		Changchun 130012
		P. R. China\\
		Emails: lwlei@jlu.edu.cn\\
		Shaoyun Shi\footnote{Corresponding author}\\
		School of Mathematics \& State key laboratory
		of automotive\\ simulation and control, Jilin University
		Changchun 130012
		P. R. China\\
		Emails: shisy@jlu.edu.cn}
	
	
	\date{}
	\maketitle
	{\noindent{\bf Abstract}} 
	This article deals with stochastic partial differential equations with quadratic nonlinearities perturbed by small additive and multiplicative noise. We present the approximate solution of the original equation
	via the amplitude equation and give the rigorous error analysis. For illustration, we apply our main theorems to stochastic Burger's equation.
	

	\vskip 2mm
	\section{Introduction}
	Stochastic partial differential equations (SPDEs) with quadratic nonlinearities of the type
	\begin{align}\label{eq000}
	\textup{d}u=\mathcal{A}u\textup{d}t+B(u,u)\textup{d}t+G(u)\textup{d}W(t),
	\end{align}
	are used to study some physical phenomenon such as hydrodynamic
	turbulence  (Burgers' equation) \cite{Ch1}, surface erosion (Kuramoto-Sivashinsky equation) \cite{La1}, amorphous thin-film growth \cite{Ra}, propagation of solitons (Korteweg-de Vries equation) \cite{Bo1, Bo2} and Rayleigh-B{\'e}nard convection \cite{Bl8}.
	
	In this paper, we consider (\ref{eq000}) perturbed by small deterministic perturbation  and small noise:
	\begin{equation}\label{eq001}
	\begin{split}
	\textup{d}u&=[\mathcal{A}u+\varepsilon^{2}\mathcal{L}u+B(u,u)]\textup{d}t+G(u,\varepsilon)\textup{d}W(t),~~u\in \mathcal{H},\\
	u(0)&=u_{0},
	\end{split}
	\end{equation}
	where  $\mathcal{H}$ is an infinite dimensional separate Hilbert space with
	scalar product $\langle\cdot,\cdot\rangle$ and corresponding norm $\|\cdot\|$, $\mathcal{A}$ is a self-adjoint and non-positive operator with finite dimensional kernel space called as dominated modes, $\mathcal{L}$ is
	a linear operator, $B$ is a bilinear and symmetric operator, $G(u,\varepsilon)$ is a Hilbert-Schmidt operator,
	$W(t)$ is a cylindrical Wiener process with covariance operator $I$ on some stochastic space,
	and $\varepsilon$ is a small parameter characterizing the distance from bifurcation point and the strength of the noise.
	
	This paper will answer two questions. One question is whether there exists a simplified system can characterize the limit behavior for the original system as $\varepsilon$ tends to $0$. This question arises naturally from the complexity of multi-scale SPDEs, which causes that it is not easy to analyze dynamical behaviors and provide numerical stimulation. After extracting effective information from the original system, we will present a reduced system, and rigorously prove that it is regarded as a good approximation of the original one. On the other hand, we will explain the extent to which stochastic forcing influences the dynamics near a change of stability. This argument is motivated by
	physics investigations \cite{Hu1,Hu2,Hu3,Ro1} in which scholars observe that noise has the potential to stabilize the dynamics. For specific systems, with the help of the simplified system, we will clearly give defined conditions under which the original system are stable or unstable.
	
	The approach we rely on is use amplitude equation deriving from dominated pattern to captures the effective dynamics of the original system.
	We precisely interpret the procedure  of the approach follows:\\
	$\bullet$ Remove high order terms from dominated modes;\\
	$\bullet$ Extract amplitude equations from dominated modes;\\
	$\bullet$ Estimate the error between the amplitude equations and the original equations.\\ Amplitude equation not only contributes to the approximation for SPDEs, but also explains
	whether the noise could shift bifurcation point. The first rigorous result for SPDEs on bounded domain via amplitude equations was established by Bl\"{o}mker \textit{et al}
	\cite{Bl2}. After that, there have been rapid progresses for SPDEs with additive noise, such as quadratic nonlinearities \cite{Bl1,Bl4}, cubic nonlinearities \cite{Bl5, Bl3}, as well as both quadratic and cubic nonlinearities \cite{Kl1, Mo2,  Mo7}. Recently, Bl\"{o}mker and Fu \cite{Fu1} considered a class of SPDEs with cubic nonlinearities perturbed by multiplicative noise via amplitude equations. However, except  results in \cite{Bl6,Bl7}, amplitude equations for SPDEs with quadratic nonlinearities perturbed by multiplicative noise are unknown, let alone additive and multiplicative noise. The aim of this paper is to develop this research. We will consider (\ref{eq001})  in two cases which  are further investigations for the results in \cite{Bl4, Mo2}. 
	
	In the first case, we are
	concerned about additive noise of order $\varepsilon^{2}$, and obtain amplitude equation without too many restrictions. 
	Compared with previous work, our result underlines the important role that multiplicative noise plays in amplitude equation. Let us illustrate this point with stochastic Burger's equation:
	\begin{align}\label{eq098}
	\textup{d}u=[(\partial_{xx}+1)u+\varepsilon^{2}\nu u+u\partial_{x}u]\textup{d}T+(\varepsilon^{2}+\varepsilon u)\textup{d}W(t)
	\end{align}
	on $[0, \pi]$ subject to Dirichlet boundary condition. Under some assumptions, we obtain amplitude equation:
	\begin{align}\label{eq099}
	\textup{d}\tilde{x}=(\nu\tilde{x}-\frac{1}{12 }\tilde{x}^{3})\textup{d}T+\alpha_{1}\textup{d}\beta_{1}(T)+\frac{8\sqrt{2}\alpha_{1}}{3\pi^{\frac{3}{2}}}\tilde{x}\textup{d}\beta_{1}(T)-\frac{8\sqrt{2}\alpha_{3}}{15\pi^{\frac{3}{2}}}\tilde{x}\textup{d}\beta_{3}(T),
	\end{align}
	where $\beta_{1}(T)$ and $\beta_{3}(T)$ are real-valued Brownian motion,  $\alpha_{1}$ and $\alpha_{3}$ are coefficients from $W(t)$.
	However, if multiplicative noise does not involve in (\ref{eq098}), under same assumptions, amplitude equation is
	\begin{align*}
	\textup{d}\tilde{x}=(\nu \tilde{x}-\frac{1}{12 }\tilde{x}^{3})\textup{d}T+\alpha_{1}\textup{d}\beta_{1}(T).
	\end{align*}
	From the comparison of two amplitude equations, multiplicative noise makes a difference to amplitude equation.
	Our main theorem further states that
	\begin{align*}
	u(t)=\varepsilon \tilde{x}(\varepsilon^{2}t)\sin{x}+\mathcal{O}(\varepsilon^{2}),
	\end{align*}
	with $\tilde{x}(\varepsilon^{-2}t)=\tilde{x}(T)$ from (\ref{eq099}).
	Then, observing (\ref{eq099}), we can clearly understand how multiplicative noise changes the stability of (\ref{eq098}) if $\alpha_{1}=0$. In fact,  the constant solution 0 is locally stable if $\nu<\frac{64\alpha_{3}^{2}}{225\pi^{3}}$, and locally unstable if $\nu>\frac{64\alpha_{3}^{2}}{225\pi^{3}}$. To our best knowledge, it is the first observation in quadratic nonlinearities with this field. More than that, such stability analysis allows us to explain that the bifurcation point of a deterministic system could change if it is perturbed by noise. Therefore, we think our result provides a new perspective for stochastic bifurcation theory.
	
	In the second one, we take into account additive noise of order $\varepsilon$, and assume that $\mathcal{L}$ commute with the projection operators. Since multiplicative noise is studied in this case, fast fluctuation appears in diffusion terms. This problem causes that we can gain amplitude equation, but we can not show how fast the solution of the  amplitude equation converges to that of the original one if the dimension of $\ker{A}$ is more than one. Fortunately, we can give the explicit error between the approximation solution and the original one for one dimensional kernel space by martingale representation theorem.
	
	The rest of the paper is organized as follows. In Section 2, we introduce some assumptions and notations. In Section 3, we study the amplitude equations of (\ref{eq001}) forced by non-degenerate noise and multiplicative noise, and present the main theorem. In Section 4, we focus on (\ref{eq001}) forced by degenerate noise and multiplicative noise, provide different amplitude equations in terms of different assumptions, and show rigorous  convergence analysis and error estimate. In Section 4. as an application of the main results
	, we study the limit behavior of stochastic Burger's equation, and discuss the relationship between the stability and noise.
	\section{Notations and Assumptions}
	Throughout the paper, notations $C$ and $C_{i}$ may denote different positive constants independent of $\varepsilon$ in different occasions. In the followings, we provide some assumptions and notations.
	\begin{assumption}\label{assu1}
		Assume that $\mathcal{A}$ is a non-positive and self-adjoint operator on $\mathcal{H}$ with eigenvalues $0=\lambda_{1}\leq\cdot\cdot\cdot\leq\lambda_{k}\cdot\cdot\cdot$,
		and $\lambda_{k}\geq Ck^{m}$ holds for all sufficiently large $k$, positive constants $m$ and $C$. Suppose that there is a complete orthonormal basis $\{e_{k}\}_{k\in\mathbb{N}}$ such that $\mathcal{A}e_{k}=-\lambda_{k}e_{k}$ and the kernel space of $\mathcal{A}$ is finite dimensional. Denote $\emph{ker}\mathcal{A}$ and the orthogonal complement
		of it by $\mathcal{N}$ and $\mathcal{S}$.
	\end{assumption}
	\begin{assumption}\label{assu1-1}
		Assume that the dimension of $\ker{\mathcal{A}}$ is $1$.
	\end{assumption}
	
	Define projections $P_{c}: \mathcal{H}\rightarrow \mathcal{N}$ and $P_{s}=I-P_{c}$.
	For a map $L$, we use $L_{c}:=P_{c}L$ and $L_{s}:=P_{s}L$.
	\begin{definition}
		For $\alpha\in\mathbb{R}$, we define the space $\mathcal{H}^{\alpha}$ as
		\begin{equation*}
		\mathcal{H}^{\alpha}=\Big\{\sum^{\infty}_{k=1}\gamma_{k}e_{k}:\sum^{\infty}_{k=1}\gamma_{k}^{2}(\lambda_{k}+1)^{\alpha}<\infty\Big\}
		\end{equation*}
		with the norm
		\begin{equation*}
		\|\sum^{\infty}_{k=1}\gamma_{k}e_{k}\|_{\alpha}=(\sum^{\infty}_{k=1}\gamma_{k}^{2}(\lambda_{k}+1)^{\alpha})^{\frac{1}{2}}.
		\end{equation*}
	\end{definition}
	\par If $\mathcal{A}$ satisfies the Assumption \ref{assu1}, it can generate an analytic semi-group $\{e^{\mathcal{A}t}\}_{t\geq 0}$ on any space $\mathcal{H}^{\alpha}$, defined by
	\begin{equation*}
	e^{\mathcal{A}t}(\sum^{\infty}_{k=1}\gamma_{k}e_{k})=\sum^{\infty}_{k=1}e^{-\lambda_{k}t}\gamma_{k}e_{k},~~ t\geq0.
	\end{equation*}
	Moreover, $\mathcal{A}$ enjoys the following property.
	\begin{lemma}\label{l51}
		Under Assumption \ref{assu1}, for $ \forall \rho\in(\lambda_{n},\lambda_{n+1}]$, $t\geq 0$, $\beta\leq\alpha$, there is a constant $M>0$,
		such that for $\forall u\in\mathcal{H}^{\beta}$,
		\begin{equation*}
		\|e^{\mathcal{A}t}P_{s}u\|_{\alpha}\leq Mt^{-\frac{\beta}{m}}e^{-\rho t}\|P_{s}u\|_{\alpha-\beta}.
		\end{equation*}
	\end{lemma}
	\begin{assumption}\label{assu2}
		Assume that $\mathcal{L}:\mathcal{H}^{\alpha}\rightarrow\mathcal{H}^{\alpha-\beta}$ is a linear continuous mapping, for some $\alpha\in\mathbb{R}$ and $\beta\in [0,m)$.
	\end{assumption}
	\begin{assumption}\label{assu3}
		Assume that $B:\mathcal{H}^{\alpha}\times\mathcal{H}^{\alpha}\rightarrow\mathcal{H}^{\alpha-\beta}$
		is a bounded bilinear and symmetric operator with $\alpha$ and $\beta$ given in Assumption \ref{assu2}.
		Moreover, suppose that $B_{c}(a)=0$, for $a\in\mathcal{N},$ where we use the notation $B(a):=B(a,a)$.
	\end{assumption}
	\begin{assumption}\label{assu3-1}
		Assume that $B_{c}(e_{k},e_{k})=0$, for $k>n$.
	\end{assumption}
	\begin{definition}
		Define $\mathcal{F}:\mathcal{N}\times\mathcal{N}\times\mathcal{N}\rightarrow\mathcal{N}$ by
		\begin{equation*}
		\mathcal{F}(u,v,w)=-B_{c}(u,\mathcal{A}_{s}^{-1}B_{s}(v,w)),~u,v,w\in\mathcal{N}.
		\end{equation*}
	\end{definition}
	\begin{assumption}\label{assu4}
		Assume that $\mathcal{F}$ is a trilinear, symmetric mapping and satisfies the following conditions:
		for positive constant $C_{0}$,
		\begin{equation}\label{eq41}
		\|\mathcal{F}(u,v,w)\|<C_{0}\|u\|\|v\|\|w\|,\quad\quad\forall u,v,w\in\mathcal{N},
		\end{equation}
		and for positive constants $C_{1}, C_{2}, C_{3}$, for all $u,v,w\in\mathcal{N}$
		\begin{align}
		\langle\mathcal{F}_{c}(u,v,w)-\mathcal{F}_{c}(v),u\rangle\leq -C_{1}\|u\|^{4}+C_{2}\|w\|^{4}+C_{3}\|w\|^{2}\|v\|^{2},\label{eq44}
		\end{align}
		where we use $\mathcal{F}(u):=\mathcal{F}(u,u,u)$ for short notation.
	\end{assumption}
	\begin{assumption}\label{assu5}
		Let $U$ be a separable Hilbert space with scalar product $\langle\cdot,\cdot\rangle_{U}$. Assume that $W(t)$ is a $U$-valued cylindrical Wiener process on a stochastic base
		$(\Omega,\mathscr{F},\{\mathscr{F}_{t}\}_{t\geq 0},\mathbb{P})$ with covariance operator $I$, the identity operator.
	\end{assumption}
	\par Note that $W(t)$ has the expansion \cite{Da}
	\begin{equation*}
	W(t)=\sum^{\infty}_{k=1}\beta_{j}(t)f_{j},
	\end{equation*}
	where $\{\beta_{j}(t)\}_{j\in\mathbb{N}}$ are real valued Brownian motions mutually independent on the above stochastic basis and $\{f_{j}\}_{\mathbb{N}}$ is a complete orthonormal system in $U$.
	
	Since multiplicative noise runs through this paper, we recall Hilbert-Schmidt operator here. Suppose that $U$, $\mathcal{H}$ are two separable Hilbert spaces with complete orthonormal basis $\{f_{j}\}_{j\in\mathbb{N}}\subset U$, $\{e_{k}\}_{k\in\mathbb{N}}\subset \mathcal{H}$. A linear and bounded operator $L$ is said to be Hilbert-Schmidt if
	\begin{equation*}
	\sum^{\infty}_{j=1}\sum^{\infty}_{k=1}|\langle L f_{j}, e_{k}\rangle_{\mathcal{H}}|^{2}<\infty.
	\end{equation*}
	Denote the set of all Hilbert-Schmidt operators from $U$ to $\mathcal{H}$ by $\mathscr{L}_{2}(U,\mathcal{H})$. Note that $\mathscr{L}_{2}(U,\mathcal{H})$
	is separable Hilbert space, with the scalar product
	\begin{equation*}
	\langle\cdot,\cdot\rangle_{\mathscr{L}_{2}(U,\mathcal{H})}=\sum^{\infty}_{j=1}\langle \cdot f_{j}, \cdot f_{j} \rangle_{\mathcal{H}},
	\end{equation*}
	by which the norm of $\mathscr{L}_{2}(U,\mathcal{H})$ is induced.
	\begin{assumption}\label{assu6}
		Suppose that $G:\mathcal{H}^{\alpha}\times \mathcal{R}^{+}\rightarrow\mathscr{L}_{2}(U,\mathcal{H}^{\alpha})$
		with $\alpha$ as in Assumption \ref{assu2},
		and $G(u,\varepsilon)=\sigma_{\varepsilon}\widetilde{G}+\varepsilon\bar{G}(u)$, where $\sigma_{\varepsilon}$ is scaling parameter and $\bar{G}(0)=0$.
		We further assume
		\begin{equation*}
		\widetilde{G}\cdot f_{k}=\alpha_{k}e_{k},
		\end{equation*}
		where $\alpha_{k}$ are real constants for $1\leq k\leq N$, and $\alpha_{k}=0$ for $k> N$.
		Moreover, assume that
		$\bar{G}(u)$ is Fr$\acute{\textsl{e}}$chet differentiable up to order 2, and for $\forall u,v,w\in\mathcal{H}^{\alpha}$ with $\|u\|_{\alpha}\leq r$,
		there exists a positive constant $l_{r}$ depending on $r$ such that
		\begin{align}
		&\|\bar{G}(u)\|_{\mathscr{L}_{2}(U,\mathcal{H}^{\alpha})}\leq l_{r}\|u\|_{\alpha},\notag \\
		&\|\bar{G}'(u)\cdot v\|_{\mathscr{L}_{2}(U,\mathcal{H}^{\alpha})}\leq l_{r}\|v\|_{\alpha}, \label{eq26} \\
		&\|\bar{G}''(u)\cdot (v,w)\|_{\mathscr{L}_{2}(U,\mathcal{H}^{\alpha})}\leq  l_{r}\|v\|_{\alpha}\|w\|_{\alpha}, \notag
		\end{align}
		where $\bar{G}'(u)$ and $\bar{G}''(u)$ are the first and second Fr$\acute{\textsl{e}}$chet derivatives with respect to $u$ respectively.
	\end{assumption}
	\begin{Remark}
		We can consider the case that $N=\infty$ and $\widetilde{G}$ is non-diagonal operator, but addition assumptions need to be imposed to ensure the convergence of various infinite series. Therefore, we ignore the case for simplicity of presentation.
	\end{Remark}
	\begin{definition}\label{def3}
		An $\mathcal{H}^{\alpha}$-valued process $u(t)$ is called the local mild solution of $(\ref{eq001})$, if there exists some stopping time $\tau_{ex}$, on a full set
		in which $\tau_{ex}>0$, we have $u\in\mathcal{C}^{0}([0,\tau_{ex}),\mathcal{H}^{\alpha}) $ and
		\begin{equation*}
		u(t)=e^{\mathcal{A}t}u(0)+\int^{t}_{0}e^{\mathcal{A}(t-s)}[\varepsilon^{2}\mathcal{L}u+B(u)]\textup{d}s+\int^{t}_{0}e^{\mathcal{A}(t-s)}G(u,\varepsilon)\textup{d}W(s),
		\end{equation*}
		$\forall t\in[0,\tau_{ex})$.
	\end{definition}
	
	According to the above assumptions, we could state that there
	exsits local mild solution in (\ref{eq001}) by cut-off technique and  Theorem 7.2 in \cite{Da}.
	\begin{Theorem}\label{theo3}
		Under Assumption $\ref{assu1}-\ref{assu6}$, for any given $u(0)\in\mathcal{H}^{\alpha}$, there exists  a unique
		local mild solution of $(\ref{eq001})$ in the sense of Definition $\ref{def3}$ such that $\tau_{ex}=\infty$
		or $\lim\limits_{t\rightarrow \tau_{ex}}\|u(t)\|_{\alpha}=\infty$.
	\end{Theorem}

	\section{Non-degenerate additive noise}
	In this section, we will consider the case that the dominated modes are affected by additive noise, and assume $\sigma_{\varepsilon}=\varepsilon^{2}$. The section is devoted to presenting the amplitude equation of $u$ and providing the approximation result. In short, we firstly bounded the fast modes. Next, we extract the amplitude equation after separating higher order terms from the slow modes in the drift and diffusion terms. Then, we show the error estimate between the solution of the amplitude equation and that of the slow modes. Finally, we prove that the solution of (\ref{eq001}) can be approximated by that of the amplitude equations well in $[0,\varepsilon^{-2}T_{0}]$.
	
	Split $u$ into slow modes $a$ and fast modes $b$:
	\begin{align*}
	u(t)=\varepsilon a(\varepsilon^{2}t)+\varepsilon^{2}b(\varepsilon^{2}t),
	\end{align*}
	where $a\in\mathcal{N}$ and $b\in\mathcal{S}$. Introduce slow time scale $T=\varepsilon^{2}t$. In what follows, we will study the behavior of (\ref{eq001}) on such scale. Then, with the projections $P_{c}$ and $P_{s}$, we have
	\begin{align}
	\textup{d}a&=[\mathcal{L}_{c}a+\varepsilon\mathcal{L}_{c}b+2B_{c}(a,b)+\varepsilon B_{c}(b)]\textup{d}T+[\widetilde{G}_{c}+\varepsilon^{-1}\bar{G}_{c}(\varepsilon a+\varepsilon^{2}b)]\textup{d}\widetilde{W}, \label{eq002} \\
	\textup{d}b&=[\varepsilon^{-2}\mathcal{A}_{s}b+\varepsilon^{-1}\mathcal{L}_{s}a+\mathcal{L}_{s}b+B_{s}(b)+2\varepsilon^{-1}B_{s}(a,b)]\textup{d}T\label{eq003}\\
	&~~+\varepsilon^{-2}B_{s}(a)\textup{d}T+[\varepsilon^{-1}\widetilde{G}_{s}+\varepsilon^{-2}\bar{G}_{s}(\varepsilon a+\varepsilon^{2}b)]\textup{d}\widetilde{W}.\notag
	\end{align}
	
	As stated in Theorem $\ref{theo3}$, the local mild solution of $(\ref{eq001})$ may blow up in finite time, so we introduce some stopping time to ensure that $a(T)$ and $b(T)$ are not too large on some interval.
	\begin{definition}\label{defs1}
		For an $\mathcal{N}\times\mathcal{S}$-valued stochastic process $(a,b)$ given by $(\ref{eq002})$ and $(\ref{eq003})$, we define, for $T_{0}>0$ and $\kappa\in(0,\frac{2}{19})$, the stopping time
		$\tau^{\star}$ as
		\begin{equation*}
		\tau^{\star}:=T_{0}\wedge\inf\{T>0\big|\|a(T)\|_{\alpha}>\varepsilon^{-\kappa}\quad or \quad\|b(T)\|_{\alpha}>\varepsilon^{-3\kappa}\}.
		\end{equation*}
	\end{definition}

	\subsection{Bounds of fast modes}
	Set $V(T):=\varepsilon^{-1}\mathcal{L}_{s}a+\mathcal{L}_{s}b+B_{s}(b)+2\varepsilon^{-1}B_{s}(a,b)+\varepsilon^{-2}B_{s}(a).$
	Writing the mild solution of $b(T)$:
	\begin{align*}
	b(T)&=e^{\varepsilon^{-2}\mathcal{A}_{s}T}b(0)+\int^{T}_{0}e^{\varepsilon^{-2}\mathcal{A}_{s}(T-s)}V(s)\textup{d}s+\varepsilon^{-1}\int^{T}_{0}e^{\varepsilon^{-2}\mathcal{A}_{s}(T-s)}\widetilde{G}_{s}\textup{d}\widetilde{W}\\
	&~~~~+\varepsilon^{-2}\int^{T}_{0}e^{\varepsilon^{-2}\mathcal{A}_{s}(T-s)}\bar{G}_{s}(\varepsilon a+\varepsilon^{2}b)\textup{d}\widetilde{W}.\\
	\end{align*}
	In the following, we will bound $b(T)$ after estimating each part of it.
	\begin{lemma}\label{lem001}
		Under Assumption \ref{assu1}, \ref{assu2}, \ref{assu3}, \ref{assu5}, \ref{assu6}, for $p>1$, there exists a positive constant $C$ such that
		\begin{equation*}
		\mathbb{E}\Big(\sup_{0\leq T\leq \tau^{\star}}  \|\int^{T}_{0}
		e^{\varepsilon^{-2}\mathcal{A}_{s}(T-s)} V(s)\textup{d}s\|_{\alpha}^{p}\Big)\leq C\varepsilon^{-2\kappa p}.
		\end{equation*}
	\end{lemma}
	\begin{proof}
		For $p>1$, by Lemma \ref{l51}, we have
		\begin{align*}
		&~~\mathbb{E}\Big(\sup_{0\leq T\leq \tau^{\star}}  \|\int^{T}_{0}
		e^{\varepsilon^{-2}\mathcal{A}_{s}(T-s)} V(s)\textup{d}s\|_{\alpha}^{p}\Big)\\
		&\leq \mathbb{E}\Big(\sup_{0\leq T\leq \tau^{\star}}  (\int^{T}_{0}
		e^{\varepsilon^{-2}\mathcal{A}_{s}(T-s)} \|V(s)\|_{\alpha}\textup{d}s)^{p}\Big)\\
		&\leq \mathbb{E}\Big(\sup_{0\leq T\leq \tau^{\star}}  (\int^{T}_{0}M(T-s)^{-\frac{\beta}{m}}e^{\rho(T-s)}\|V(s)\|_{\alpha-\beta}\textup{d}s)^{p}\Big)\\
		&\leq C\varepsilon^{2p}\mathbb{E}\Big( \sup_{0\leq  T\leq \tau^{\star}}\|V(T)\|^{p}_{\alpha-\beta}\Big).
		\end{align*}
		Then, based on Assumption \ref{assu2}-\ref{assu3}, it is easy to obtain the desired result.
		
		We complete the proof.
	\end{proof}
	
	Let $\mathcal{Z}(T)=\varepsilon^{-1}\int^{T}_{0}e^{\varepsilon^{-2}\mathcal{A}_{s}(T-s)}\widetilde{G}_{s}\textup{d}\widetilde{W}=\sum\limits^{N}_{k=n+1}\mathcal{Z}_{k}(T)e_{k}$,
	where
	\begin{equation*}
	\mathcal{Z}_{k}(T)=\varepsilon^{-1}\alpha_{k}\int^{T}_{0}e^{-\varepsilon^{-2}\lambda_{k}(T-s)}\textup{d}\tilde{\beta}_{k}(s)
	\end{equation*}
	with $\tilde{\beta}_{k}(s):=\varepsilon\beta_{k}(\varepsilon^{-2}s)$.
	Then, according to Lemma 14 in \cite{Bl3}, we have the following lemma.
	\begin{lemma}\label{l03}
		Under Assumption \ref{assu1}, \ref{assu5}, \ref{assu6}, for $\mathcal{Z}(T)$ and $\mathcal{Z}_{k}(T)$ given by the above, there exists a positive constant $C$ depending on $p\geq1, \lambda_{k}, \alpha_{k}, \kappa>0$ and $T_{0}$, such that
		\begin{align}
		&\mathbb{E}\Big(\sup_{0\leq T\leq T_{0}}|\mathcal{Z}_{k}(t)|^{p}\Big)\leq C\varepsilon^{-\frac{\kappa}{3}},\label{eq19}\\
		&\mathbb{E}\Big(\sup_{0\leq T\leq T_{0}}\|\mathcal{Z}(t)\|_{\alpha}^{p}\Big)\leq C\varepsilon^{-\frac{\kappa}{3}}.\label{eq20}
		\end{align}
	\end{lemma}
	\begin{lemma}\label{lem002}
		Under Assumption \ref{assu1}, \ref{assu2}, \ref{assu3}, \ref{assu5}, \ref{assu6}, for $p>1$, there exists a positive constant $C$ such that
		\begin{equation*}
		\mathbb{E}\Big(\sup_{0\leq T\leq \tau^{\star}}  \|
		\int^{T}_{0}e^{\varepsilon^{-2}\mathcal{A}_{s}(T-s)}\bar{G}_{s}(\varepsilon a+\varepsilon^{2}b)\textup{d}\widetilde{W}\|_{\alpha}^{p}\Big)\leq C\varepsilon^{2-2\kappa p}.
		\end{equation*}
	\end{lemma}
	\begin{proof}
		We complete the proof with the help of  factorization method.
		For $p>2$, we choose $\gamma\in (\frac{1}{p},\frac{1}{2})$, and introduce
		\begin{align}
		\mathcal{D}(T)=\int^{T}_{0}(T-s)^{-\gamma}e^{\mathcal{A}_{s}(T-s)\varepsilon^{-2}}\bar{G}_{s}(\varepsilon a(s)+\varepsilon^{2}b(s))\textup{d}\tilde{W}.
		\end{align}
		By Stochastic Fubini Theorem, we know
		\begin{align*}
		\int^{T}_{0}e^{\mathcal{A}_{s}(T-s)\varepsilon^{-2}}\bar{G}_{s}(\varepsilon a(s)+\varepsilon^{2}b(s))\textup{d}\tilde{W}=C_{\gamma}\int^{T}_{0}(T-s)^{\gamma-1}e^{\mathcal{A}_{s}(T-s)\varepsilon^{-2}}\mathcal{D}(s)\textup{d}s,
		\end{align*}
		where $C_{\gamma}$ is a constant dependent of $\gamma$.
		By Lemma \ref{l51} and H\"{o}lder inequality, we further deduce
		\begin{align*}
		&~~\|\int^{T}_{0}e^{\mathcal{A}_{s}(T-s)\varepsilon^{-2}}\bar{G}_{s}(\varepsilon a(s)+\varepsilon^{2}b(s))\textup{d}\tilde{W}\|_{\alpha}^{p}\\
		&\leq C\Big(\int^{T}_{0}(T-s)^{\gamma-1}e^{-\rho(T-s)\varepsilon^{-2}}\|\mathcal{D}(s)\|_{\alpha}\textup{d}s\Big)^{p}\\
		&\leq C\Big(\int^{T}_{0}(T-s)^{\frac{p(\gamma-1)}{p-1}}e^{\frac{-\rho(T-s)\varepsilon^{-2}p}{p-1}}\textup{d}s\Big)^{p-1}\int^{T}_{0}\|\mathcal{D}(s)\|_{\alpha}^{p}\textup{d}s\\
		&\leq C\varepsilon^{2\gamma p-2}\int^{T}_{0}\|\mathcal{D}(s)\|_{\alpha}^{p}\textup{d}s.
		\end{align*}	
		By Lemma \ref{l51}, Burkholder-Davis-Gundy inequality and Assumption \ref{assu5}, we obtain the moments of $\mathcal{D}(T)$ as follows:
		\begin{align*}
		\mathbb{E}\|\mathcal{D}(T)\|_{\alpha}^{p}&\leq C\mathbb{E}\Big(\int^{T}_{0}(T-s)^{-2\gamma}\|e^{\mathcal{A}_{s}(t-s)\varepsilon^{-2}       }\bar{G}_{s}(\varepsilon a+\varepsilon^{2}b)\|^{2}_{\mathscr{L}_{2}(U,\mathcal{H}^{\alpha})  }\textup{d}s\Big)^{\frac{p}{2}}\\
		&\leq C\mathbb{E}\Big(\int^{T}_{0}(T-s)^{-2\gamma}e^{-\rho(t-s)\varepsilon^{-2}       }\mathbb{I}_{[0,\tau^{\star}]}(s)\|\varepsilon a+\varepsilon^{2}b)\|^{2}_{\alpha}\textup{d}s\Big)^{\frac{p}{2}}\\
		&\leq C\varepsilon^{p-2\gamma p}\mathbb{E}\sup_{0\leq T\leq T^{\star}}\|\varepsilon a(T)+\varepsilon^{2}b(T)\|_{\alpha}^{p}\\
		&\leq C\varepsilon^{2p-2\gamma p-\kappa p}.
		\end{align*}
		We conclude
		\begin{align*}
		\mathbb{E}\Big(\sup_{0\leq  T\leq \tau^{\star}}\|\int^{T}_{0}e^{\mathcal{A}_{s}(T-s)\varepsilon^{-2}}\bar{G}_{s}(\varepsilon a(s)+\varepsilon^{2}b(s))\textup{d}\tilde{W}\|_{\alpha}^{p}\Big)\leq C\varepsilon^{2-\kappa p-2}.
		\end{align*}
		Then, we can achieve this proof by H\"{o}lder inequality,.
		
		We complete the proof.
	\end{proof}
	
	Combining Lemma \ref{lem001}-\ref{lem002}, we give the bound of $b(T)$ by triangle inequality.
	\begin{lemma}\label{lem13}
		Under Assumption \ref{assu1}, \ref{assu2}, \ref{assu3}, \ref{assu5}, \ref{assu6}, for $p>1$, there exists a positive constant $C$ such that
		\begin{equation*}
		\mathbb{E}\Big(\sup_{0\leq T\leq \tau^{\star}}\|b(T)\|^{p}\Big)\leq C\|b(0)\|^{p}+C\varepsilon^{-2\kappa p}.
		\end{equation*}
	\end{lemma}
	\subsection{Amplitude equation}
	We now turn to obtain the amplitude equation by effective information from the dominant part $a(T)$. Observing (\ref{eq001}), we notice that $\mathcal{A}_{s}b=B_{s}(a,a)+$ high order terms, so one attempt to
	replace $b$ by $\mathcal{A}_{s}^{-1}B_{s}(a,a)$ in the drift part of (\ref{eq002}). As for the diffusion term, we choose the linearization of $\bar{G}$ and neglect information about $b(T)$.
	In what follows, we make this intuitive idea rigorous.
	
	Applying It$\hat{\textup{o}}$'s formula to $B_{c}(a,\mathcal{A}_{s}^{-1}b)$, we rewrite $a(T)$ as
	\begin{equation}\label{eq011}
	a(T)=a(0)+\int^{T}_{0}[\mathcal{L}_{c}a+2\mathcal{F}(a)]\textup{d}s+\int^{T}_{0}[\widetilde{G}_{c}+\bar{G}_{c}'(0)a]\textup{d}\widetilde{W}+R(T),
	\end{equation}
	where
	\begin{align*}
	R(T)
	&=2\varepsilon^{2}B_{c}(a(T),\mathcal{A}_{s}^{-1}b(T))-2\varepsilon^{2}B_{c}(a(0),\mathcal{A}_{s}^{-1}b(0))+\varepsilon\int_{0}^{T}\mathcal{L}_{s}b\textup{d}s+\varepsilon\int_{0}^{T}B_{c}(b,b)\textup{d}s\\
	&\quad-2\varepsilon^{2}\int_{0}^{T}B_{c}(\mathcal{L}_{c}a,\mathcal{A}_{s}^{-1}b)\textup{d}s-2\varepsilon^{3}\int_{0}^{T}B_{c}(\mathcal{L}_{c}b,\mathcal{A}_{s}^{-1}b)\textup{d}s\\
	&\quad-2\varepsilon^{3}\int_{0}^{T}B_{c}(B_{c}(b,b),\mathcal{A}_{s}^{-1}b)\textup{d}s-4\varepsilon^{2}\int_{0}^{T}B_{c}(B_{c}(a,b),\mathcal{A}_{s}^{-1}b)\textup{d}s\\
	&\quad -2\varepsilon\int_{0}^{T}(B_{c}(a,\mathcal{A}_{s}^{-1}\mathcal{L}_{s}a))\textup{d}s -2\varepsilon\int_{0}^{T}(B_{c}(a,\mathcal{A}_{s}^{-1}\mathcal{L}_{s}b))\textup{d}s\\
	&\quad-2\varepsilon^{2}\int_{0}^{T}B_{c}(a,\mathcal{A}_{s}^{-1}B_{s}(b,b))\textup{d}s-4\varepsilon\int_{0}^{T}B_{c}(a,\mathcal{A}_{s}^{-1}B_{s}(a,b))\textup{d}s\\
	&\quad-2\int^{T}_{0}\sum^{\infty}_{j=1}B_{c}(\bar{G}_{c}(\varepsilon a+\varepsilon^{2}b)f_{j}, \mathcal{A}_{s}^{-1}\widetilde{G}_{s}f_{j}))\textup{d}s\\
	&\quad-2\varepsilon^{-1}\int_{0}^{T}\sum^{\infty}_{j=1}B_{c}(\bar{G}_{c}(\varepsilon a+\varepsilon^{2}b)f_{j}, \mathcal{A}_{s}^{-1}\bar{G}_{s}(\varepsilon a+\varepsilon^{2}b)f_{j})\textup{d}s\\
	&\quad-2\int^{T}_{0}\sum^{\infty}_{j=1}B_{c}(\widetilde{G}_{c}f_{j}, \mathcal{A}_{s}^{-1}\bar{G}_{s}(\varepsilon a+\varepsilon^{2}b)f_{j}    )\textup{d}s-2\varepsilon\int^{T}_{0}\sum^{\infty}_{j=1}B_{c}(\widetilde{G}_{c}f_{j}, \mathcal{A}_{s}^{-1}\widetilde{G}_{s}f_{j})\textup{d}s\\
	&\quad+\int^{T}_{0}[\frac{1}{\varepsilon}\bar{G}_{c}(\varepsilon a+\varepsilon^{2}b)-\bar{G}_{c}'(0)a]\textup{d}\widetilde{W}-2\varepsilon^{2}\int^{T}_{0}B_{c}(\widetilde{G}_{c}\textup{d}\widetilde{W}, \mathcal{A}_{s}^{-1}b)\\
	&\quad-2\varepsilon\int_{0}^{T}B_{c}(\bar{G}_{c}(\varepsilon a+\varepsilon^{2}b)\textup{d}\widetilde{W},\mathcal{A}_{s}^{-1}b)-2\varepsilon\int^{T}_{0}B_{c}(a, \mathcal{A}^{-1}_{s}\widetilde{G}_{s}\textup{d}\widetilde{W})\\
	&\quad-2\int_{0}^{T}B_{c}(a, \mathcal{A}_{s}^{-1}\bar{G}_{s}(\varepsilon a+\varepsilon^{2}b)\textup{d}\widetilde{W}).\\
	\end{align*}
	Let us show the bounds of $R(T)$.
	\begin{lemma}\label{lemr}
		Under Assumption \ref{assu1}, \ref{assu2}, \ref{assu3}, \ref{assu5}, \ref{assu6}, for $p>1$, there exists a positive constant $C$ such that
		\begin{equation}\label{eq010}
		\mathbb{E}\Big(\sup_{0\leq T\leq \tau^{\star}}\|R(T)\|_{\alpha}^{p}\Big)\leq C \varepsilon^{p-9\kappa p}.
		\end{equation}
	\end{lemma}
	\begin{proof}
		Notice that all $\mathcal{H}^{\alpha}$ norms are equivalent on $\mathcal{N}$, and $\mathcal{A}_{s}^{-1}$ is a bounded linear operator from $\mathcal{H}_{\alpha-1}$ to $\mathcal{H}_{\alpha}$.
		Then, we obtain
		\begin{align*}
		&~~\mathbb{E}\Big(\sup_{0\leq T\leq \tau^{\star}}\|\varepsilon
		^{2}\int_{0}^{T}B_{c}(\mathcal{L}_{c}a,\mathcal{A}_{s}^{-1}b)\textup{d}s\|_{\alpha}^{p}\Big)\\
		&\leq C \mathbb{E}\Big(\varepsilon^{2}\int_{0}^{\tau}\|B_{c}(\mathcal{L}_{c}a,\mathcal{A}_{s}^{-1}b)\|_{\alpha}\textup{d}s\Big)^{p}\\
		&\leq C\varepsilon^{2p}\|a\|_{\alpha}\|b\|_{\alpha}\\
		&\leq C\varepsilon^{2p-4\kappa}.
		\end{align*}
		We can estimate other drift terms by the similar calculation.
		
		According to Taylor Formula, we have
		\begin{align}\label{eq0104}
		\bar{G}(\varepsilon a+\varepsilon^{2}b)=\bar{G}(0)+\varepsilon\bar{G}'(0)(a+b)+\bar{G}''(\theta)(\varepsilon a+\varepsilon^{2}b,\varepsilon a+\varepsilon^{2}b),
		\end{align}
		where $\theta$ is in the line connecting $0$ and $\varepsilon a+\varepsilon^{2}b$.\\
		Then, we apply Burkholder-Davis-Gundy inequality and (\ref{eq0104}) to derive
		\begin{align*}
		&\mathbb{E}\Big(\sup_{0\leq T\leq \tau^{\star}}\|\int^{T}_{0}[\frac{1}{\varepsilon}\bar{G}_{c}(\varepsilon a+\varepsilon^{2}b)-\bar{G}_{c}'(0)a]\textup{d}\widetilde{W}\|_{\alpha}^{p}\Big)\\
		\leq&C\mathbb{E}\Big(\sup_{0\leq T\leq \tau^{\star}}\|\frac{1}{\varepsilon}\bar{G}_{c}(\varepsilon a+\varepsilon^{2}b)-\bar{G}_{c}'(0)a\|^{p}_{\mathcal{L}_{2}(U,\mathcal{H}^{\alpha})}\Big)\\
		\leq&C\varepsilon^{p-3\kappa p}.
		\end{align*}
		Moreover, by Burkholder-Davis-Gundy inequality, we have
		\begin{align*}
		&\mathbb{E}\Big(\sup_{0\leq T\leq \tau^{\star}}\|\varepsilon\int_{0}^{T}B_{c}(\bar{G}_{c}(\varepsilon a+\varepsilon^{2}b)\textup{d}\widetilde{W},\mathcal{A}_{s}^{-1}b)\|_{\alpha}^{p}\Big)\\
		\leq&C\varepsilon^{p}\mathbb{E}\Big(\int_{0}^{\tau^{\star}}\sum^{\infty}_{j=1}\|B_{c}(\bar{G}_{c}(\varepsilon a+\varepsilon^{2}b)f_{j},\mathcal{A}_{s}^{-1}b)\|_{\alpha}^{2}\textup{d}s\Big)^{\frac{p}{2}}\\
		\leq&C\varepsilon^{p}\mathbb{E}\Big(\int_{0}^{\tau^{\star}}\sum^{\infty}_{j=1}\|B_{c}(\bar{G}_{c}(\varepsilon a+\varepsilon^{2}b)f_{j},\mathcal{A}_{s}^{-1}b)\|_{\alpha-\beta}^{2}\textup{d}s\Big)^{\frac{p}{2}}\\
		\leq&C\varepsilon^{p}\mathbb{E}\Big(\sup_{0\leq T \leq \tau^{\star}}\sum^{\infty}_{j=1}\|B_{c}(\bar{G}_{c}(\varepsilon a+\varepsilon^{2}b)f_{j},\mathcal{A}_{s}^{-1}b)\|_{\alpha-\beta}^{2}\Big)^{\frac{p}{2}}\\
		\leq&C\varepsilon^{p-3\kappa p}\mathbb{E}\Big(\sup_{0\leq T \leq \tau^{\star}}\|\bar{G}(\varepsilon a+\varepsilon^{2}b)\|^{p}_{\mathcal{L}_{2}(U,\mathcal{H}^{\alpha})}\Big)\\
		\leq&C\varepsilon^{2p-6\kappa p}.
		\end{align*}
		Similarly, we can estimate other diffusion terms, but the detail is not provided here.

		Collecting all the estimates, we own (\ref{eq010}).
		
		We complete the proof.
	\end{proof}
	
	Removing the $R(T)$ from (\ref{eq011}), we gain the amplitude equations as
	\begin{equation}\label{eq012}
	\begin{split}
	\textup{d}x&=[\mathcal{L}_{c}x+2\mathcal{F}(x)]\textup{d}T+[\widetilde{G}_{c}+\bar{G}_{c}'(0)x]\textup{d}\widetilde{W}\\
	x(0)&=a(0).
	\end{split}
	\end{equation}
	\subsection{Rigorous error analysis}
	Now let us provide the bounds of $x(T)$, and give the better estimate of $a(T)$ via $x(T)$.
	\begin{lemma}\label{boundb}
		Let Assumption \ref{assu1}, \ref{assu2}, \ref{assu3}, \ref{assu4}-\ref{assu6} hold. For $p>1$, there exists a positive constant $C$ such that
		\begin{align}
		&\mathbb{E}\Big(\sup_{0\leq T \leq T_{0}} \|x(T)\|^{p}\Big)\leq C\|a(0)\|^{p}+C.\label{eq301}\\
		\intertext{Moreover, if $\|a(0)\|\leq \varepsilon^{-\frac{\kappa}{3}}$,}
		&\mathbb{E}\Big(\sup_{0\leq T\leq \tau^{\star}}\|a(T)-x(T)\|^{p}\Big)\leq C\varepsilon^{p-18\kappa p},\notag\\
		&\mathbb{E}\Big(\sup_{0\leq T\leq \tau^{\star}}\|a(T)\|^{p}\Big)\leq C\varepsilon^{-\frac{\kappa p}{3}}.\notag
		\end{align}
	\end{lemma}
	\begin{proof}
		Define some stopping time
		\begin{equation*}
		\tau_{K}:=\inf\{T>0, \|x(T)\|>K\}.
		\end{equation*}
		
		For $p\geq 2$ and $0 \leq T\leq \tau_{K}$, It$\hat{\textup{o}}^{,}$s formula yields that
		\begin{equation*}
		\begin{split}
		\|x(T)\|^{p}&\leq\|a(0)\|^{p}+p\int^{T}_{0}\|x\|^{p-2}\langle\mathcal{L}_{c}x,x\rangle\textup{d}s+2p\int^{T}_{0}\|x\|^{p-2}\langle\mathcal{F}_{c}(x),x\rangle\textup{d}s\\
		&\quad+p\int^{T}_{0}\|x\|^{p-2}\langle x,[\widetilde{G}_{c}+\bar{G}'_{c}(0)x]\textup{d}\widetilde{W}(s)\rangle+Cp(p-1)\int^{T}_{0}(\|x\|^{p}+\|x\|^{p-2})\textup{d}s.
		\end{split}
		\end{equation*}
		For $T_{1}\in[0,T_{0}]$, by Cauchy-Schwarz inequality, Burkholder-Davis-Gundy inequality and Young's inequality we deduce that
		\begin{equation*}
		\begin{split}
		&\mathbb{E}\Big(\sup_{0\leq T\leq T_{1}\wedge\tau_{K}}\|x(T)\|^{p}\Big)\\
		\leq&C\|a(0)\|^{p}+C+C\mathbb{E}\Big(\int^{T_{1}\wedge\tau_{K}}_{0}\|x(s)\|^{p}\textup{d}s\Big)+C\mathbb{E}\Big(\int^{T_{1}\wedge\tau_{K}}_{0}\|x(s)\|^{2p}\textup{d}s\Big)^{\frac{1}{2}}\\
		\leq &C\|a(0)\|^{p}+C+C\int^{T_{1}}_{0}\mathbb{E}\Big(\sup_{0\leq s_{1}\leq s\wedge\tau_{K}}\|x(s_{1})\|^{p}\Big)\textup{d}s\frac{1}{2}\mathbb{E}\Big(\sup_{0\leq T\leq T_{1}\wedge\tau_{K}}\|x(T)\|^{p}\Big).
		\end{split}
		\end{equation*}
		Using Gronwall's lemma, we own
		\begin{equation}\label{eq081}
		\mathbb{E}\Big(\sup_{0\leq T\leq T_{0}\wedge\tau_{K}}\|x(T)\|^{p}\Big)\leq C\|a(0)\|^{p}+C.
		\end{equation}
		We claim (\ref{eq081}) still holds if $T_{0}\wedge\tau_{K}$ is replaced with $T_{0}$.\\
		In fact. Let the right side of (\ref{eq081}) be $\widetilde{C}$.
		Then,
		\begin{equation*}
		\begin{split}
		\quad\mathbb{P}\Big(\tau_{K}>T_{0}\Big)&=\mathbb{P}\Big(\sup_{T\leq T_{0} \wedge\tau_{K}}\|x(T)\|<K\Big)\\
		&=1-\mathbb{P}\Big(\sup_{T\leq T_{0}\wedge\tau_{K}}\|x(T)\|\geq K\Big)\\
		&\geq 1-\frac{\widetilde{C}}{K^{p}}.
		\end{split}
		\end{equation*}
		Then, we know that 
		$\sup\limits_{0\leq T\leq T_{0}\wedge\tau_{K}}\|x(t)\|^{p}$ monotonously converges to $\sup\limits_{0\leq T\leq T_{0}}\|x(t)\|^{p}~a.s.$, as $K\rightarrow \infty$.
		Thus, monotone convergence theorem yields
		\begin{equation*}
		\mathbb{E}\Big(\sup_{0\leq T \leq T_{0}} \|x(T)\|^{p}\Big)=\lim_{K\rightarrow \infty} \mathbb{E}\Big(\sup_{0\leq T \leq T_{0}\wedge\tau_{K}}\|x(T)\|^{p}\Big)\leq \widetilde{C}.
		\end{equation*}
		For $1\leq p<2$, H$\textup{\"{o}}$lder inequality yields (\ref{eq301}).
		
		Let $h(T):=x(T)-a(T)+R(T)$. Then, we have
		\begin{equation}
		\begin{split}
		h(T)&=\int^{T}_{0}\mathcal{L}_{c}(h-R)\textup{d}s+2\int^{T}_{0}\mathcal{F}_{c}(x)\textup{d}s-2\int^{T}_{0}\mathcal{F}_{c}(x-h+R)\textup{d}s\\
		&\quad+\int^{T}_{0}\varepsilon^{-1}\bar{G}'_{c}(0)(h-R)\textup{d}\widetilde{W}(s).
		\end{split}
		\end{equation}
		For $p\geq 2$, using It$\hat{\textup{o}}^{,}$s Lemma, we obtain
		\begin{align*}
		\|h(T)\|^{p}&\leq C\int^{T}_{0}\|h\|^{p-2}\langle\mathcal{L}_{c}(h-R),h\rangle\textup{d}s+C\int^{T}_{0}\|h\|^{p-2}\langle\mathcal{F}_{c}(x)-\mathcal{F}_{c}(x-h+R),h\rangle\textup{d}s\\
		&~~+C\int^{T}_{0}\|h\|^{p-2}\langle h,\bar{G}'_{c}(0)(h-R)\textup{d}\tilde{W}\rangle+C\int^{T}_{0}\|h\|^{p-2}\|h-R\|^{2}\textup{d}s.
		\end{align*}
		According to condition (\ref{eq44}), Cauchy-Schwarz inequality and Young's inequality, we obtain
		\begin{align*}
		\|h(T)\|^{p}&\leq C\int^{T}_{0}(\|h\|^{p}+\|R\|^{p}+\|R\|^{2p}+\|R\|^{p}\|x\|^{p})\textup{d}s\\
		&~~+C\int^{T}_{0}\|h\|^{p-2}\langle h,\bar{G}'_{c}(0)(h-R)\textup{d}\tilde{W}\rangle+C\int^{T}_{0}\|h\|^{p-2}\|h-R\|^{2}\textup{d}s.
		\end{align*}
		Furthermore, for any $T^{\star}\in [0,T_{0}]$, applying to Burkholder-Davis-Gundy inequality and Young's inequality, we have
		\begin{align*}
		\mathbb{E}(\sup_{0\leq T\leq T^{\star}\wedge \tau^{\star}}\|h(T)\|^{p})&\leq C\mathbb{E}(\sup_{0\leq T\leq T^{\star}\wedge \tau^{\star}}\int^{T}_{0}\|h\|^{p}+\|R\|^{p}+\|R\|^{2p}+\|R\|^{p}\|x\|^{p}\textup{d}s)\\
		&~~+C\mathbb{E}\big(\int^{T^{\star}\wedge T_{0}}_{0}(\|h\|^{2p}+\|R\|^{2p})\textup{d}s\big)^{\frac{1}{2}}\\
		&\leq \frac{1}{2}\mathbb{E}(\sup_{0\leq T\leq T^{\star}\wedge\tau^{\star}}\|h(T)\|^{p})+
		C\int^{T}_{0}\mathbb{E}(\sup_{0\leq s_{1}\leq s\wedge \tau^{\star}}\|h(s_{1})\|^{p})\textup{d}s\\
		&~~+C\varepsilon^{p-18\kappa p}.
		\end{align*}
		Gronwall's lemma yields
		\begin{equation*}
		\begin{split}
		\mathbb{E}\Big(\sup_{0\leq T\leq\tau^{\star}}\|h(T)\|^{p}\Big)\leq C \varepsilon^{p-18\kappa p}.
		\end{split}
		\end{equation*}
		Then, by triangle inequality, we derive
		\begin{equation*}
		\begin{split}
		\mathbb{E}\Big(\sup_{0\leq T\leq \tau^{\star}}\|a(T)-x(T)\|^{p}\Big)&\leq \mathbb{E}\Big(\sup_{0\leq T\leq \tau^{\star}}\|h(T)\|^{p}\Big)+\mathbb{E}\Big(\sup_{0\leq T\leq \tau^{\star}}\|R(T)\|^{p}\Big)\\
		&\leq C\varepsilon^{p-18\kappa p},
		\end{split}
		\end{equation*}
		and
		\begin{equation*}
		\begin{split}
		\mathbb{E}\Big(\sup_{0\leq T\leq \tau^{\star}}\|a(T)\|^{p}\Big)&\leq \mathbb{E}\Big(\sup_{0\leq T\leq \tau^{\star}}\|a(T)-x(T)\|^{p}\Big)+\mathbb{E}\Big(\sup_{0\leq T\leq \tau^{\star}}\|x(T)\|^{p}\Big)\\
		&\leq C\varepsilon^{-\frac{\kappa p}{3}}.
		\end{split}
		\end{equation*}
		
		We complete the proof.
	\end{proof}
	
	Although most of  previous estimates hold on random interval $[0,\tau^{\star}]$, we will show the main theorem holds on fixed interval $[0,T_{0}/\varepsilon^{2}]$.
	\begin{lemma}\label{lem15}
		Let \ref{assu1}, \ref{assu2}, \ref{assu3}, \ref{assu5}, \ref{assu6} hold. For $p>1$ and $\|u(0)\|_{\alpha}\leq -\frac{\kappa}{3}$,
		there exists a positive constant $C$ such that
		\begin{equation*}
		\mathbb{E}\Big(\sup_{0\leq T\leq \tau^{\star}}\|\mathcal{R}(T)\|^{p}_{\alpha}\Big)\leq C\varepsilon^{2p-18\kappa p},
		\end{equation*}
		where $\mathcal{R}(T)=u(\varepsilon^{-2}T)-\varepsilon x(T)$.
	\end{lemma}
	\begin{proof}
		Note that $
		\mathcal{R}(T)=u(\varepsilon^{-2}T)-\varepsilon x(T)=\varepsilon a(T)+\varepsilon^{2}b(T)-\varepsilon x(T)$.
		Then, we can prove this Lemma by Lemma \ref{lem13} and Lemma \ref{boundb}, .
		
		We complete the proof.
	\end{proof}
	\begin{definition}
		For $\kappa>0$ , define $\Omega^{\star}\subset\Omega$ of all $\omega\subset\Omega$ satisfy that all the following estimations
		\begin{equation*}
		\sup_{0\leq T \leq \tau^{\star}}\|a(T)\|< \varepsilon^{-\frac{\kappa}{2}},~~
		\sup_{0\leq T \leq \tau^{\star}}\|b(T)\|_{\alpha}<\varepsilon^{-\frac{5\kappa}{2}},~~
		\sup_{0\leq T \leq \tau^{\star}}\|\mathcal{R}(T)\|_{\alpha}< \varepsilon^{2-19\kappa}.
		\end{equation*}
		
	\end{definition}
	\begin{lemma}\label{lem16}
		Under Assumption \ref{assu1}, \ref{assu2}, \ref{assu3}, \ref{assu4}-\ref{assu6}, for $p>1$, there exists a positive constant $C$ such that
		\begin{equation*}
		\mathbb{P}(\Omega^{\star})\geq 1-C\varepsilon^{p}.
		\end{equation*}
	\end{lemma}
	\begin{proof}
		For fixed $p>1$, using Lemma \ref{lem13}, Lemma \ref{boundb}, Lemma \ref{lem15} and Chebyshev inequality, we have
		\begin{equation*}
		\begin{split}
		\mathbb{P}(\Omega^{\star})&\geq 1-\mathbb{P}(\sup_{0\leq T \leq \tau^{\star}}\|a(T)\|\geq\varepsilon^{-\frac{\kappa}{2}})-\mathbb{P}(\sup_{0\leq T \leq \tau^{\star}}\|b(T)\|_{\alpha}\geq\varepsilon^{-\frac{5\kappa}{2}})\\
		&\quad-\mathbb{P}(\sup_{0\leq T \leq \tau^{\star}}\|\mathcal{R}(T)\|_{\alpha}\geq\varepsilon^{2-19\kappa})\\
		&\geq 1-\varepsilon^{\frac{\kappa q}{2}}\mathbb{E}(\sup_{0\leq T \leq \tau^{\star}}\|a(T)\|^{q})-\varepsilon^{\frac{5\kappa q}{2}}\mathbb{E}(\sup_{0\leq T \leq \tau^{\star}}\|b(T)\|^{q}_{\alpha})\\
		&\quad-\varepsilon^{19\kappa q-2q}\mathbb{E}(\sup_{0\leq T \leq \tau^{\star}}\|\mathcal{R}(T)\|^{q}_{\alpha})\\
		&\geq 1-C\varepsilon^{p},
		\end{split}
		\end{equation*}
		where $q$ is large enough.
		
		We complete the proof.
	\end{proof}
	\begin{Theorem}\label{the1}
		Let Assumption \ref{assu1}, \ref{assu2}, \ref{assu3}, \ref{assu4}-\ref{assu6} hold. Let $u(t)$ be the local mild solution of (\ref{eq001}) with $\|u(0)\|_{\alpha}\leq \varepsilon^{1-\frac{\kappa}{3}}$.
		Then, for any $p>1$, there exists a positive constant $C$ such that
		\begin{equation*}
		\mathbb{P}\Big(\sup_{0\leq t\leq\varepsilon^{-2}T_{0}}\|u(t)-\varepsilon x(\varepsilon^{2}t)\|_{\alpha}>\varepsilon^{2-19\kappa}\Big)\leq \varepsilon^{p}.
		\end{equation*}
	\end{Theorem}
	\begin{proof}
		Note that
		\begin{equation*}
		\Omega^{\star}\subseteq\Big\{\omega\Big|\sup_{0\leq T \leq \tau^{\star}}\|a(T)\|< \varepsilon^{-\kappa},
		\sup_{0\leq T \leq \tau^{\star}}\|b(T)\|_{\alpha}<\varepsilon^{-3\kappa}\Big\}\subseteq\{\omega\Big|\tau^{\star}=T_{0}\}\subseteq\Omega.
		\end{equation*}
		Then,
		\begin{equation*}
		\sup_{0\leq T \leq T_{0}}\|\mathcal{R}(T)\|_{\alpha}=\sup_{0\leq T \leq \tau^{\star}}\|\mathcal{R}(T)\|_{\alpha}< \varepsilon^{2-19\kappa}, \omega\in \Omega^{\star}.
		\end{equation*}
		Lemma \ref{lem16} implies that
		\begin{equation*}
		\mathbb{P}(\sup_{0\leq T \leq T_{0}}\|\mathcal{R}(T)\|_{\alpha}\geq\varepsilon^{2-19\kappa})\leq 1-\mathbb{P}(\Omega^{\star})\leq \varepsilon^{p}.
		\end{equation*}
		
		We complete the proof.	
	\end{proof}
	\begin{Remark}
		We comment Theorem \ref{the1} covers the case that \ref{eq001} is forced either additive noise or multiplicative one. In particular, if the amplitude equation is only with multiplicative noise, we can analyze the stability of the trivial solution, and further study how multiplicative noise changes the stability. We will give an example to illustrate it in Section 5. 	
	\end{Remark}
	\section{Degenerate noise}
	In this section, we are committed to (\ref{eq001}) with $\sigma_{\varepsilon}=\varepsilon$ and degenerate additive noise (i.e., the noise does not influence the dominant modes directly). Moreover, we suppose that $\mathcal{L}$ commute with $P_{c}$ and $P_{s}$.

	Because the additive noise is of order $\varepsilon$ in this section, the same decomposition as Section 3 does not allow us to control the fast modes. In order to overcome this trouble, we will use another decomposition to consider (\ref{eq001}).
	
	Let $u(t)=\varepsilon \varphi(\varepsilon^{2}t)+\varepsilon \psi(\varepsilon^{2}t)$, where $\varphi\in\mathcal{N}$ and $\psi\in\mathcal{S}$. Introduce slow time scale $T=\varepsilon^{2}t$.
	By projections $P_{c}$ and $P_{s}$, we split $u$ into
	\begin{align}
	\textup{d}\varphi&=[\mathcal{L}_{c}\varphi+2\varepsilon^{-1}B_{c}(\varphi,\psi)+\varepsilon^{-1} B_{c}(\psi)]\textup{d}T+\varepsilon^{-1}\bar{G}_{c}(\varepsilon \varphi+\varepsilon \psi)\textup{d}\widetilde{W},  \label{eqx01} \\
	\textup{d}\psi&=[\varepsilon^{-2}\mathcal{A}_{s}\psi+\mathcal{L}_{s} \psi+\varepsilon^{-1}B_{s}(\psi)+2\varepsilon^{-1}B_{s}(\varphi,\psi)+\varepsilon^{-1}B_{s}(\varphi)]\textup{d}T\label{eqy01}\\
	&~~+[\varepsilon^{-1}\widetilde{G}+\varepsilon^{-1}\bar{G}_{s}(\varepsilon \varphi+\varepsilon \psi)]\textup{d}\widetilde{W}.\notag
	\end{align}
	
	As Section 2, we also hope to replace $\psi(T)$ by $\varphi(T)$ in (\ref{eqx01}) and obtain amplitude solution. Although we can gain a reduced system after replacing $\psi(T)$ by $\varphi(T)$, there exists still O-U process in the diffusion part of it, which causes it is perturbed by fast fluctuation. We will respectively treat this problem in two cases. The first case is the dimension of $\ker{A}$ is more than one. In this case, we can deal with the O-U process and obtain the amplitude equation, but we just can show the law of the solution of the amplitude equation weakly converges to that of $\varphi(T)$ without explicit convergence rate. The other is $\ker{A}$ is one-dimensional space. In this case, we can not only obtain the amplitude equation, but also present the precise error between the approximation solution and the original one.
	
	The section is divided into three subsections. In Subsection 4.1, we will estimate $\psi(T)$ and obtain a reduced system with O-U process. In Subsection 4.2, our aim is to study the amplitude equation in case that $\ker{A}$ is multi-dimensional space. In Subsection 4.3, we will show the amplitude equation  for one-dimensional kernel space, and give rigorous error
	analysis.
	\subsection{Reduced system}\label{reduced}
	
	In order to guarantee that $\varphi(T)$ and $\psi(T)$ do not get out of a bounded domain on some interval , we introduce a stopping time.
	\begin{definition}\label{defs2}
		For an $\mathcal{N}\times\mathcal{S}$-valued stochastic process $(\varphi,\psi)$ given by $(\ref{eqx01})$ and $(\ref{eqy01})$, we define, for $T_{0}>0$ and $\kappa\in(0,\frac{1}{13})$, the stopping time
		$\tau^{\star}_{1}$ as
		\begin{equation*}
		\tau^{\star}_{1}:=T_{0}\wedge\inf\{T>0\big|\|\varphi(T)\|_{\alpha}>\varepsilon^{-\kappa}\quad or \quad\|\psi(T)\|_{\alpha}>\varepsilon^{-\kappa}\}.
		\end{equation*}
	\end{definition}
	For simplicity of representation, we introduce a notation to express higher order term
	in the sense of probability.
	\begin{definition}\label{defso2}
		Let  $\{X_{\varepsilon}(T)\}_{T\geq0}$ be
		a family of real-valued processes, we say $X_{\varepsilon}=\bar{\mathcal{O}}(f_{\varepsilon})$ with respect to stopping time $\tau^{\star}_{1}$, if for every $p>1$ there exists a positive constant $C$ such that
		\begin{equation*}
		\mathbb{E}\Big(\sup_{0\leq T\leq\tau^{\star}_{1}}|X_{\varepsilon}(T)|^{p}\Big)\leq Cf^{p}_{\varepsilon}.
		\end{equation*}
	\end{definition}
	
	Letting $V_{1}(T)=\mathcal{L}_{s} \psi+\varepsilon^{-1}B_{s}(\psi)+2\varepsilon^{-1}B_{s}(\varphi,\psi)+\varepsilon^{-1}B_{s}(\varphi)$, we present the mild solution of $\psi(T)$:
	\begin{align*}
	\psi(T)&=e^{\varepsilon^{-2}\mathcal{A}_{s}T}\psi(0)+\int^{T}_{0}e^{\varepsilon^{-2}\mathcal{A}_{s}(T-s)}V_{1}(s)\textup{d}s+\varepsilon^{-1}\int^{T}_{0}e^{\varepsilon^{-2}\mathcal{A}_{s}(T-s)}\widetilde{G}_{s}\textup{d}\widetilde{W}\\
	&~~~~+\varepsilon^{-1}\int^{T}_{0}e^{\varepsilon^{-2}\mathcal{A}_{s}(T-s)}\bar{G}_{s}(\varepsilon \varphi+\varepsilon \psi)\textup{d}\widetilde{W}\\
	&:=Q(T)+J(T)+\mathcal{Z}(T)+K(T).
	\end{align*}
	
	\begin{lemma}\label{l02} Under Assumption \ref{assu1}, \ref{assu2}, \ref{assu3}, \ref{assu5}, \ref{assu6}, for $p>1$, there exists a constant positive $C$ such that
		\begin{align*}
		\mathbb{E}\Big(\sup_{0\leq T\leq\tau^{\star}_{1}}\|J(T)\|_{\alpha}^{p}\Big)\leq C\varepsilon^{p-2\kappa p}.
		\end{align*}
	\end{lemma}
	\begin{proof}
		This proof is similar to that of Lemma \ref{lem001}.
		
		We complete the proof.
	\end{proof}

	\begin{lemma}\label{l30}
		Under Assumption \ref{assu1}, \ref{assu2}, \ref{assu3}, \ref{assu5}, \ref{assu6}, for $p>1$, there exists a positive constant $C$ such that
		\begin{equation}
		\mathbb{E}\Big(\sup_{0\leq T\leq\tau^{\star}_{1}}\|K(T)\|_{\alpha}^{p}\Big)\leq C\varepsilon^{p-2\kappa p}.
		\end{equation}
	\end{lemma}
	\begin{proof}
		This proof is similar to that of Lemma \ref{lem002}.
		
		We complete the proof.
	\end{proof}
	
	Combining Lemma \ref{l03}, \ref{l02}, \ref{l30}, we can estimate $\psi(T)$ by triangle inequality.
	\begin{lemma}\label{l07}
		Under Assumption \ref{assu1}, \ref{assu2}, \ref{assu3}, \ref{assu5}, \ref{assu6}, for $p>1$ , there exists positive constant $C$ such that
		\begin{equation*}
		\mathbb{E}\Big(\sup_{0\leq T\leq \tau^{\star}_{1}}\|\psi(T)\|^{p}\Big)\leq C\|\psi(0)\|^{p}+C\varepsilon^{\frac{-\kappa p}{3}}.
		\end{equation*}
	\end{lemma}
	Now, we begin to replace $\psi(T)$ and $\varphi(T)$.
	\begin{lemma}\label{l1}
		Under Assumption \ref{assu1}, \ref{assu2}, \ref{assu3}, \ref{assu5}, \ref{assu6}, we have
		\begin{align*}
		\int^{T}_{0}B_{c}(\varphi,\psi)\textup{d}s&=-2\varepsilon\int^{T}_{0}B_{c}(B_{c}(\varphi,\mathcal{Z}), \mathcal{A}_{s}^{-1}\mathcal{Z})\textup{d}s-\varepsilon\int^{T}_{0}B_{c}(B_{c}(\mathcal{Z}), \mathcal{A}_{s}^{-1}\mathcal{Z})\textup{d}s \\
		&~~~~-\varepsilon\int^{T}_{0}B_{c}(\varphi,\mathcal{A}_{s}^{-1}B_{s}(\varphi))\textup{d}s-2\varepsilon\int^{T}_{0}B_{c}(\varphi,\mathcal{A}_{s}^{-1}B_{s}(\varphi,\mathcal{Z}))\textup{d}s \\
		&~~~~-\varepsilon\int^{T}_{0}B_{c}(\varphi,\mathcal{A}_{s}^{-1}B_{s}(\mathcal{Z})\textup{d}s-\varepsilon\int^{T}_{0}B_{c}(\varphi,\mathcal{A}_{s}^{-1}\widetilde{G}\textup{d}\widetilde{W} )\\
		&~~~~-\varepsilon\int^{T}_{0}\sum^{\infty}_{j=1}B_{c}(\bar{G}_{c}'(0)(\varphi+\mathcal{Z})f_{j}, \mathcal{A}^{-1}_{s}\widetilde{G}f_{j})\textup{d}s+R_{1}(T),
		\end{align*}
		with $R_{1}(T)=\bar{\mathcal{O}}(\varepsilon^{2-6\kappa})$.
	\end{lemma}
	\begin{proof}
		Since the proof is fairly standard, we will omit some straightforward computations.
		
		By It$\hat{\textup{o}}$ formula, we derive that
		\begin{align*}
		\textup{d}B_{c}(\varphi,\mathcal{A}_{s}^{-1}\psi)&=B_{c}(\mathcal{L}_{c} \varphi,\mathcal{A}_{s}^{-1}\psi)\textup{d}T+2\varepsilon^{-1}B_{c}(B_{c}(\varphi,\psi),\mathcal{A}_{s}^{-1}\psi)\textup{d}T+\varepsilon^{-2}B_{c}(\varphi,\psi)\textup{d}T\\
		&~~~~+\varepsilon^{-1}B_{c}(B_{c}(\psi),\mathcal{A}_{s}^{-1}\psi)\textup{d}T+\varepsilon^{-1}B_{c}(\varphi,\mathcal{A}_{s}^{-1}B_{s}(\varphi+\psi)    )\textup{d}T\\
		&~~~~+B_{c}(\varphi,\mathcal{L}_{s}\mathcal{A}_{s}^{-1}\psi)\textup{d}T+\varepsilon^{-2}\sum^{\infty}_{j=1}B_{c}(\bar{G}_{c}(\varepsilon \varphi+\varepsilon \psi)f_{j}, \mathcal{A}_{s}^{-1}\widetilde{G}f_{j})\textup{d}T\\
		&~~~~+\varepsilon^{-2}\sum^{\infty}_{j=1} B_{c}(\bar{G}_{c}(\varepsilon \varphi+\varepsilon \psi)f_{j}, \mathcal{A}_{s}^{-1}\bar{G}_{s}(\varepsilon \varphi+\varepsilon \psi)f_{j})\textup{d}T\\
		&~~~~+\varepsilon^{-1}B_{c}(\varphi, \mathcal{A}_{s}^{-1}\widetilde{G}\textup{d}\widetilde{W})+\varepsilon^{-1}B_{c}(\varphi, \mathcal{A}_{s}^{-1}\bar{G}_{s}(\varepsilon \varphi+\varepsilon \psi)\textup{d}\widetilde{W})\\
		&~~~~+\varepsilon^{-1}B_{c}(\bar{G}_{c}(\varepsilon \varphi+\varepsilon \psi)\textup{d}\widetilde{W}, \mathcal{A}^{-1}_{s}\psi).	
		\end{align*}
		Firstly, let us prove that
		\begin{align}\label{eq080}
		\int^{T}_{0}B_{c}(B_{c}(\varphi,\psi),\mathcal{A}^{-1}_{s}\psi)\textup{d}s=\int^{T}_{0}B_{c}(B_{c}(\varphi,\mathcal{Z}),\mathcal{A}^{-1}_{s}\mathcal{Z})\textup{d}s+\bar{\mathcal{O}}(\varepsilon^{1-5\kappa }).
		\end{align}
		Recalling the mild solution of $\psi$, we have
		\begin{align*}
		&~~~~\int^{T}_{0}B_{c}(B_{c}(\varphi,\psi),\mathcal{A}^{-1}_{s}\psi)\textup{d}s-\int^{T}_{0}B_{c}(B_{c}(\varphi,\mathcal{Z}),\mathcal{A}^{-1}_{s}\mathcal{Z})\textup{d}s\\
		&=\int^{T}_{0}B_{c}(B_{c}(\varphi,Q),\mathcal{A}^{-1}_{s}Q)\textup{d}s+\int^{T}_{0}B_{c}(B_{c}(\varphi,Q),\mathcal{A}^{-1}_{s}(J+K))\textup{d}s\\
		&~~~~+\int^{T}_{0}B_{c}(B_{c}(\varphi,Q),\mathcal{A}^{-1}_{s}\mathcal{Z})\textup{d}s+\int^{T}_{0}B_{c}(B_{c}(\varphi,J+K),\mathcal{A}^{-1}_{s}Q)\textup{d}s\\
		&~~~~+\int^{T}_{0}B_{c}(B_{c}(\varphi,J+K),\mathcal{A}^{-1}_{s}(J+K))\textup{d}s+\int^{T}_{0}B_{c}(B_{c}(\varphi,J+K),\mathcal{A}^{-1}_{s}\mathcal{Z})\textup{d}s\\
		&~~~~+\int^{T}_{0}B_{c}(B_{c}(\varphi,\mathcal{Z}),\mathcal{A}^{-1}_{s}Q)\textup{d}s+\int^{T}_{0}B_{c}(B_{c}(\varphi,\mathcal{Z}),\mathcal{A}^{-1}_{s}(J+K))\textup{d}s
		\end{align*}
		Estimate the first term of the right hand of the above equation as follow:
		\begin{align*}
		&~~~~\mathbb{E}\Big(\sup_{0\leq T\leq \tau^{\star}_{1}}\|\int^{T}_{0}B_{c}(B_{c}(\varphi,Q),\mathcal{A}^{-1}_{s}Q)\textup{d}s\|_{\alpha}^{p}\Big)\\
		&\leq C\mathbb{E}\Big(\sup_{0\leq T\leq \tau^{\star}_{1}}(\int^{T}_{0}\|B_{c}(B_{c}(\varphi,Q),\mathcal{A}^{-1}_{s}Q)\|_{\alpha}\textup{d}s)^{p}\Big)\\
		&\leq C\varepsilon^{-\kappa p}\Big(\int^{T_{0}}_{0}\|e^{\mathcal{A}_{s}s\varepsilon^{-2}}\varphi(0)\|_{\alpha}^{2}\textup{d}s\Big)\\
		&\leq C\varepsilon^{2p-\frac{5\kappa p}{3}}.
		\end{align*}
		By similar technique, it is easy to obtain (\ref{eq080}) by Lemma \ref{l02}-\ref{l30}.\\
		Furthermore, we can prove that
		\begin{align*}
		&\int^{T}_{0}B_{c}(B_{c}(\psi),\mathcal{A}_{s}^{-1}\psi)\textup{d}s=\int^{T}_{0}B_{c}(B_{c}(\mathcal{Z}), \mathcal{A}_{s}^{-1}\mathcal{Z})\textup{d}s+\bar{\mathcal{O}}(\varepsilon^{1-6\kappa }),\\
		&\int^{T}_{0}B_{c}(\varphi,\mathcal{A}_{s}^{-1}B_{s}(\varphi+\psi))\textup{d}s=\int^{T}_{0}B_{c}(\varphi,\mathcal{A}_{s}^{-1}B_{s}(\varphi+\mathcal{Z}))\textup{d}s+\bar{\mathcal{O}}(\varepsilon^{1-5\kappa }),\\
		&\int^{T}_{0}B_{c}(B_{c}(\psi),\mathcal{A}_{s}^{-1}\psi)\textup{d}s=\int^{T}_{0}B_{c}(B_{c}(\mathcal{Z}), \mathcal{A}_{s}^{-1}\mathcal{Z})\textup{d}s+\bar{\mathcal{O}}(\varepsilon^{1-6\kappa }).
		\end{align*}
		Then, by Assumption \ref{assu6} and Taylor formula, we obtain
		\begin{align*}
		&\int^{T}_{0}\sum^{\infty}_{j=1}[B_{c}(\bar{G}_{c}(\varepsilon \varphi+\varepsilon \psi)f_{j}, \mathcal{A}_{s}^{-1}\widetilde{G}f_{j})-B_{c}(\bar{G}_{c}'(0)(\varphi+\mathcal{Z})f_{j}, \mathcal{A}^{-1}_{s}\widetilde{G}f_{j})]\textup{d}s=\bar{\mathcal{O}}(\varepsilon^{2-2\kappa }),
		\\
		&\int^{T}_{0}\sum^{\infty}_{j=1} B_{c}(\bar{G}_{c}(\varepsilon \varphi+\varepsilon \psi)f_{j}, \mathcal{A}_{s}^{-1}\bar{G}_{s}(\varepsilon \varphi+\varepsilon \psi)f_{j})\textup{d}s=\bar{\mathcal{O}}(\varepsilon^{2-2\kappa }),
		\end{align*}
		by Assumption \ref{assu6} and Burkholder-Davis-Gundy inequality, we obtain the stochastic terms are $\bar{\mathcal{O}}(\varepsilon^{1-2\kappa })$,
		and by simple calculation, we obtain other terms are $\bar{\mathcal{O}}(\varepsilon^{2-2\kappa })$.
		
		We complete the proof.
	\end{proof}

	To analyze $B_{c}(\psi)$ precisely, we need to introduce several notations.	
	For a Hilbert Space $\mathcal{H}$, denote the tensor product of it by $v_{1}\otimes v_{2}$, and the symmetric tensor product of it by $v_{1}\otimes_{s}v_{2}=\frac{1}{2}\big(v_{1}\otimes v_{2}+v_{1}\otimes v_{2}\big)$, where $v_{1}, v_{2}\in \mathcal{H}$. Furthermore, define the scalar product  in the
	tensor product space $\mathcal{H}_{\alpha}\otimes \mathcal{H}_{\beta}$ by $\langle u_{1}\otimes v_{1}, u_{2}\otimes v_{2} \rangle_{\alpha,\beta}:=\langle u_{1}, u_{2}\rangle_{\alpha}\langle v_{1}, v_{2} \rangle_{\beta}$, where $u_{1}, u_{2}\in \mathcal{H}_{\alpha}, v_{1}, v_{2}\in \mathcal{H}_{\beta}$. For simplified notation, we adopt $\langle,\rangle_{\alpha}$ instead of $\langle,\rangle_{\alpha,\alpha}$. The norm of $\mathcal{H}_{\alpha}\otimes\mathcal{H}_{\beta}$ is induced by such scalar product.
	For two linear operators $\mathcal{L}_{a}$ and $\mathcal{L}_{b}$ on $\mathcal{H}$, define the symmetric tensor product of them by $\big(\mathcal{L}_{a}\otimes_{s}\mathcal{L}_{b}\big)(v_{1}\otimes v_{2})=\frac{1}{2}
	\big(\mathcal{L}_{a} v_{1}\otimes \mathcal{L}_{b} v_{2}+\mathcal{L}_{b}v_{2}\otimes\mathcal{L}_{a}v_{1}\big).$
	\begin{lemma}\label{l2}
		Under Assumption \ref{assu1}, \ref{assu2}, \ref{assu3}, \ref{assu3-1}, \ref{assu5}, \ref{assu6}, we have
		\begin{align*}
		\int^{T}_{0}B_{c}(\psi)\textup{d}s&=-\varepsilon\int^{T}_{0}B_{c}(I\otimes_{s}\mathcal{A}_{s})^{-1}\Big(\mathcal{Z}\otimes_{s}B_{s}(\varphi+\mathcal{Z})\Big)\textup{d}s\\
		&~~-\varepsilon\int^{T}_{0}\sum^{N}_{j=n+1}\alpha_{j}B_{c}(I\otimes_{s}\mathcal{A}_{s})^{-1}\Big(e_{i}\otimes_{s}\bar{G}_{s}'(0)(\varphi+\mathcal{Z})f_{j}\Big)\textup{d}s\\
		&~~-\varepsilon\int^{T}_{0}B_{c}(I\otimes_{s}\mathcal{A}_{s})^{-1}\Big(\mathcal{Z}\otimes_{s}\widetilde{G}\textup{d}\widetilde{W}\Big)+R_{2}(T),
		\end{align*}
		with $R_{2}(T)=\bar{\mathcal{O}}(\varepsilon^{2-6\kappa})$.
	\end{lemma}
	
	\begin{proof}
		Applying It$\hat{\textup{o}}$ formula to $\psi\otimes \psi$, we obtain
		\begin{align*}
		\frac{1}{2}\textup{d}(\psi\otimes \psi)&= \varepsilon^{-2}(\psi\otimes_{s}\mathcal{A}_{s}\psi)\textup{d}T+(\psi\otimes_{s} \mathcal{L}_{s}\psi)\textup{d}T+\varepsilon^{-1} (\psi\otimes_{s}B_{s}(\varphi+\psi))\textup{d}T\\
		&~~+\varepsilon^{-1}(\psi\otimes_{s} \widetilde{G}\textup{d}\widetilde{W})+\varepsilon^{-1}(\psi\otimes_{s} \bar{G}_{s}(\varepsilon \varphi+\varepsilon \psi)\textup{d}\widetilde{W})\\
		&~~+\varepsilon^{-2}\sum^{\infty}_{j=1}\bar{G}_{s}(\varepsilon \varphi+\varepsilon \psi)f_{j}\otimes \bar{G}_{s}(\varepsilon \varphi+\varepsilon \psi)f_{j}\textup{d}T\\ &~~+\varepsilon^{-2}\sum_{j=n+1}^{N}\alpha_{j}^{2}(e_{j}\otimes e_{j})\textup{d}T+\varepsilon^{-2}\sum^{N}_{j=n+1}\alpha_{j}(e_{j}\otimes_{s}\bar{G}_{s}(\varepsilon \varphi+\varepsilon \psi)f_{j})\textup{d}T.
		\end{align*}
		Note that $(I\otimes_{s} \mathcal{A}_{s})^{-1}$ is bounded operator from $\mathcal{H}_{\alpha-1}\otimes\mathcal{H}_{\alpha-1}$ to $\mathcal{H}_{\alpha}\otimes\mathcal{H}_{\alpha}$.
		The remaining proof is similar to that of Lemma \ref{l1}, so it is omitted.
		
		We complete the proof.
	\end{proof}

	\begin{lemma}\label{l21}
		Under Assumption \ref{assu1}, \ref{assu2}, \ref{assu3}, \ref{assu5}, \ref{assu6}, we have
		\begin{equation*}
		\int^{T}_{0}\varepsilon^{-1}\bar{G}_{c}(\varepsilon \varphi+\varepsilon \psi)\textup{d}\tilde{W}=\int^{T}_{0}\bar{G}'_{c}(0)(\varphi+\mathcal{Z})\textup{d}\tilde{W}+R_{3}(T),
		\end{equation*}
		with $R_{3}(T)=\bar{\mathcal{O}}(\varepsilon^{1-3\kappa })$
	\end{lemma}
	\begin{proof}
		The proof mainly relies on Taylor formula, and is similar to Lemma \ref{l1}, so we do not present the detail.
		
		We complete the proof.
	\end{proof}
	
	Although we can replace $\psi(T)$ by $\varphi(T)$ in slow modes by previous lemmas, one still need to eliminate the
	fast Ornstein-Uhlenbeck process $\mathcal{Z}(T)$ appearing in the drift and diffusion terms. Therefore, we will present some useful lemmas to deal with such problem.
	We introduce some notations before showing them.
	
	Set $\hat{\mathcal{Z}}(T):=\sum\limits^{N}_{k=n+1}\hat{\mathcal{Z}}_{k}(T)$, where $\hat{\mathcal{Z}}_{k}(T)=e^{-\lambda_{k}\varepsilon^{-2}T}\hat{\mathcal{Z}}_{k}(0)+\mathcal{Z}_{k}(T)$,
	with $\hat{\mathcal{Z}}_{k}(0)$ is the normal distribution $N(0, \frac{\alpha_{k}^{2}}{2\lambda_{k}})$.  Set $\hat{G}:=\sum\limits^{N}_{k=n+1}\frac{\alpha_{k}^{2}}{2\lambda_{k}}(e_{k}\otimes e_{k})$.
	
	\begin{lemma}\textup{\cite{Bl1}}
		For every $p>0$ and $\alpha>0$, there exists a positive constant $C$ such that the bounds
		\begin{align*}
		&\mathbb{E}\Big\|\int^{T}_{s}\hat{\mathcal{Z}}(r)\textup{d}r\Big\|_{\alpha}^{2p}\leq C(T-s)^{p}\varepsilon^{2p},\\
		&\mathbb{E}\Big\|\int^{T}_{s}\big(\hat{\mathcal{Z}}(r)\otimes\hat{\mathcal{Z}}(r)-\hat{G})\textup{d}r\big\|_{\alpha}^{2p}\leq C(T-s)^{p}\varepsilon^{2p},\\
		&\mathbb{E}\Big\|\int^{T}_{s}\big(\hat{\mathcal{Z}}(r)\otimes\hat{\mathcal{Z}}(r)\otimes\hat{\mathcal{Z}}(r)\big)\textup{d}r\Big\|_{\alpha}^{2p}\leq C(T-s)^{p}\varepsilon^{2p},
		\end{align*}
		hold for every $T>s>0$.
	\end{lemma}

	\begin{lemma}\label{l3}\textup{\cite{Bl1}}
		Let $\alpha$ be as in Assumption \ref{assu2}. Let $f_{i}$ with $i\in \{1,2,3\}$ be $\tilde{\alpha}$-H$\ddot{o}$lder continuous functions on $[0,\tau^{\star}_{1}]$ with values in $\big((\mathcal{H}^{\alpha})^{\otimes i }\big)^{\star}$, respectively. Let $F_{\varepsilon}$ be given by
		\begin{align*}
		F_{\varepsilon}(t):=\int^{t}_{0}\Big(
		\big(f_{1}(s)\big)(\hat{\mathcal{Z}})+\big(f_{2}(s)\big)(\hat{\mathcal{Z}}\otimes\hat{\mathcal{Z}}-\hat{G})+\big(f_{3}(s)\big)(\hat{\mathcal{Z}}\otimes\hat{\mathcal{Z}}\otimes\hat{\mathcal{Z}})\textup{d}s
		\Big).
		\end{align*}
		Then, for $p>0$ and every $0<\gamma<\frac{2\tilde{\alpha}}{1+2\tilde{\alpha}}$ there exists a positive constant $C$ depending only on $p$ and $\gamma$ such that
		\begin{align*}
		\mathbb{E}\sup_{t\in[0,\tau^{\star}_{1}]}|F_{\varepsilon}(t)|^{p}\leq C\varepsilon^{\gamma p}\big(
		\mathbb{E}\big(\|f_{1}\|_{C^{\tilde{\alpha}}}+\|f_{2}\|_{C^{\tilde{\alpha}}}+\|f_{3}\|_{C^{\tilde{\alpha}}}\big)^{2p}
		\big)^{\frac{1}{2}}
		\end{align*}
		where $\|\cdot\|_{C^{\tilde{\alpha}}}$ denotes the $\tilde{\alpha}$-H$\ddot{o}$lder norm for
		$\big((\mathcal{H}^{\alpha})^{\otimes i}\big)^{\star}$-valued functions on $[0,\tau^{\star}]$.
	\end{lemma}
	
	With the help of Lemma \ref{l3}, we can average the fast O-U process in the drift term of $\varphi(T)$, and give the next lemma.
	\begin{lemma}\label{lx}
		Under Assumption \ref{assu1}, \ref{assu2}, \ref{assu3}, \ref{assu3-1}, \ref{assu5}, \ref{assu6}, we have
		\begin{align*}
		\varphi(T)=\varphi(0)+\bar{\mathcal{L}} \varphi+2\int^{T}_{0}\mathcal{F}(\varphi)\textup{d}s+\int^{T}_{0}\Gamma(\varphi,\hat{\mathcal{Z}})\textup{d}\widetilde{W}+R_{\varphi}(T),
		\end{align*}
		where $\bar{\mathcal{L}}$ is a linear operator from $\mathcal{H}^{\alpha}$ to $\mathcal{H}^{\alpha-\beta}$ satisfying
		\begin{align*}
		\bar{\mathcal{L}}\varphi&=\mathcal{L}_{c} \varphi-\sum^{N}_{i=n+1}\frac{2\alpha_{i}^{2}}{\lambda_{i}^{2}}B_{c}(B_{c}(\varphi,e_{i}), e_{i})-\sum^{N}_{i=n+1}\frac{\alpha_{i}^{2}}{\lambda_{i}}B_{c}(\varphi,\mathcal{A}_{s}^{-1}B_{s}(e_{i},e_{i}))\\
		&~~-2\sum^{\infty}_{j=1}B_{c}(\bar{G}_{c}'(0)(\varphi)f_{j}, \mathcal{A}^{-1}_{s}\widetilde{G}f_{j})-\sum^{N}_{i=n+1}\frac{\alpha_{i}^{2}}{\lambda_{i}}B_{c}(I\otimes_{s}\mathcal{A}_{s})^{-1}(e_{i}\otimes_{s}B_{s}(\varphi,e_{i}))\\
		&~~-\sum^{N}_{j=n+1}\alpha_{j}B_{c}(I\otimes_{s}\mathcal{A}_{s})^{-1}(e_{j}\otimes_{s}\bar{G}_{s}'(0)(\varphi)f_{j}),
		\end{align*}
		$\Gamma(\cdot,\hat{\mathcal{Z}})$ is a operator from $\mathcal{H}^{\alpha}$ to $\mathscr{L}_{2}(U,\mathcal{H}^{\alpha})$ satisfying
		\begin{align*}
		\Gamma(\varphi,\hat{\mathcal{Z}})=\bar{G}'_{c}(0)(\varphi+\hat{\mathcal{Z}})\cdot-2B_{c}(\varphi,\mathcal{A}_{s}^{-1}\widetilde{G}\cdot)
		-B_{c}(I\otimes_{s}\mathcal{A}_{s})^{-1}(\hat{\mathcal{Z}}\otimes_{s}\widetilde{G}\cdot).
		\end{align*}
		and $R_{\varphi}(T)=\bar{\mathcal{O}}(\varepsilon^{\frac{1}{5}-6\kappa})$.
		
		Moreover, for any $T_{1}, T_{2}\in[0,T_{0}]$, there exists a positive constant $C$ such that
		\begin{equation}\label{eq82}
		\mathbb{E}\|R_{\varphi}(T_{1}\wedge\tau^{\star})-R_{\varphi}(T_{2}\wedge\tau^{\star})\|^{2p}\leq C|T_{1}-T_{2}|^{p}.
		\end{equation}
	\end{lemma}
	\begin{proof}
		Firstly, according to Lemma \ref{l1}-\ref{l21}, we have
		\begin{align}\label{eq086}
		\varphi(T)&=\varphi(0)+\int^{T}_{0}\hat{\mathcal{F}}(\varphi,\mathcal{Z})\textup{d}s+\int^{T}_{0}\Gamma(\varphi,\mathcal{Z})\textup{d}\widetilde{W}+\sum_{k=1}^{3}R_{k}(T),
		\end{align}
		where
		\begin{align*}
		\hat{\mathcal{F}}(\varphi,\mathcal{Z})&=\mathcal{L}_{c}\varphi+2\mathcal{F}(\varphi)-4B_{c}(B_{c}(\varphi,\mathcal{Z}), \mathcal{A}_{s}^{-1}\mathcal{Z})-2B_{c}(B_{c}(\mathcal{Z},\mathcal{Z}), \mathcal{A}_{s}^{-1}\mathcal{Z}) \\
		&~~-4B_{c}(\varphi,\mathcal{A}_{s}^{-1}B_{s}(\varphi,\mathcal{Z}))-2B_{c}(\varphi,\mathcal{A}_{s}^{-1}B_{s}(\mathcal{Z},\mathcal{Z}))\\
		&~~-2\sum^{\infty}_{j=1}B_{c}(\bar{G}_{c}'(0)(\varphi+\mathcal{Z})f_{j}, \mathcal{A}^{-1}_{s}\widetilde{G}f_{j})\\
		&~~-\sum^{N}_{j=n+1}\alpha_{j}B_{c}(I\otimes_{s}\mathcal{A}_{s})^{-1}(e_{j}\otimes_{s}\bar{G}_{s}'(0)(\varphi+\mathcal{Z})f_{j}).
		\end{align*}
		Based on H$\ddot{\textup{o}}$lder inequality and Burkholder-Davis-Gundy inequality, we note that there exists a positive constant $C_{1}$ such that
		\begin{align}\label{eq089}
		\mathbb{E}\big(\sup_{0\leq s<T\leq\tau^{\star}_{1}}\frac{\|\sum\limits^{3}_{k=1}(R_{k}(T)-R_{k}(s))\|_{\alpha}}{(T-s)^{\tilde{\alpha}}}\big)\leq C_{2}\varepsilon^{1-6\kappa},
		\end{align}
		with $0<\tilde{\alpha}<\frac{1}{2}$. For simplify, we choose $\tilde{\alpha}=\frac{1}{3}$ in the following part of this paper.
		Thus, there exists a positive constant $C_{2}$, such that
		\begin{align}\label{eq090}
		\mathbb{E}\big(\sup_{0\leq s<T\leq\tau^{\star}_{1}}\frac{\|\varphi(T)-\varphi(s)\|_{\alpha}}{(T-s)^{\frac{1}{3}}}\big)\leq C_{2}\varepsilon^{-6\kappa}.
		\end{align}
		Combining $\varphi(T)=\bar{\mathcal{O}}(\varepsilon^{-3\kappa})$ with (\ref{eq091}), we get $\|\varphi(T)\|_{C^{\frac{1}{3}}}=\bar{\mathcal{O}}(\varepsilon^{-6\kappa})$.\\
		Next, replacing $\mathcal{Z}$ with $\hat{\mathcal{Z}}$ in (\ref{eq086}), we obtain
		\begin{align}\label{eq091}
		\varphi(T)&=\varphi(0)+\int^{T}_{0}\hat{\mathcal{F}}(\varphi,\hat{\mathcal{Z}})\textup{d}s+\int^{T}_{0}\Gamma(\varphi,\hat{\mathcal{Z}})\textup{d}\widetilde{W}+\sum^{4}_{k=1}R_{k}(T).
		\end{align}
		Here we just estimate one term in $R_{4}(T)$ as follow:
		\begin{align*}
		&~~\mathbb{E}\Big(\sup_{0\leq T\leq \tau^{\star}_{1}}\|\int^{T}_{0}B_{c}(\varphi,\mathcal{A}_{s}^{-1}B_{s}(\varphi,\mathcal{Z}-\hat{\mathcal{Z}}))\textup{d}s\|^{p}\Big)\\
		&\leq C\varepsilon^{-2\kappa p}\sum^{N}_{k=n+1}\mathbb{E}\|\int^{T_{0}}_{0}e^{-\lambda_{k}\varepsilon^{-2}s}\hat{\mathcal{Z}}_{k}(0)\textup{d}s\|^{p}\\
		&\leq C\varepsilon^{2p-2\kappa p}.
		\end{align*}
		We can estimate other terms by similar calculation, and obtain
		$R_{6}(T)=\bar{\mathcal{O}}(\varepsilon^{2-2\kappa})$.
		Then, applying Lemma \ref{l3} to (\ref{eq091}) and choosing $\gamma=\frac{1}{5}$, we can obtain $R_{\varphi}=\mathcal{O}(\varepsilon^{\frac{1}{5}-6\kappa})$.
		
		Obviously, (\ref{eq82}) is a consequence of simple calculation. In fact, we can prove that for $p>0$,
		\begin{align*}
		&~~\mathbb{E}\|\int^{T_{1}\wedge\tau^{\star}_{1}}_{T_{2}\wedge\tau^{\star}_{1}}B_{c}(\varphi,\mathcal{A}_{s}^{-1}B_{s}(\varphi,\mathcal{Z}-\hat{\mathcal{Z}}))\textup{d}s\|^{2p}\\
		&\leq C\varepsilon^{-2\kappa}\sum^{N}_{k=n+1}\mathbb{E}\|\int^{T_{1}}_{T_{2}}e^{-\lambda_{k}\varepsilon^{-2}s}\hat{\mathcal{Z}}_{k}(0)\textup{d}s\|^{2p}\\
		&\leq C|T_{1}-T_{2}|^{p},
		\end{align*}
		and
		\begin{align*}
		&~~\mathbb{E}\|\int^{T_{1}\wedge\tau^{\star}_{1}}_{T_{2}\wedge\tau^{\star}_{1}}B_{c}(\varphi,\mathcal{A}_{s}^{-1}B_{s}(\varphi,\hat{\mathcal{Z}}))\textup{d}s\|^{2p}\\
		&\leq C\varepsilon^{-2\kappa}\sum^{N}_{k=n+1}\mathbb{E}\|\int^{T_{1}}_{T_{2}}\hat{\mathcal{Z}}_{k}(s)\textup{d}s\|^{2p}\\
		&\leq C|T_{1}-T_{2}|^{p},
		\end{align*}
		We can also estimate other terms in $R_{\varphi}$ by similar technique, but the detail is omitted.
		
		We complete the proof.
	\end{proof}
	
	Remove $R_{\varphi}(T)$ from $\varphi(T)$, we gain the reduced system:
	\begin{equation}  \label{eq0130}
	\begin{split}
	\textup{d}y_{1}&=[\bar{\mathcal{L}} y_{1}+2\mathcal{F}(y_{1})]\textup{d}T+\Gamma(y_{1},\hat{\mathcal{Z}})\textup{d}\widetilde{W},\\
	y_{1}(0)&=\varphi(0).
	\end{split}
	\end{equation}
	Observing (\ref{eq0130}), there exists $\hat{\mathcal{Z}}$ in the diffusion term. We will get rid of it in next two subsections. Our next task is to give the bound of  $y_{1}(T)$ and estimate the error between $\varphi(T)$ and $y_{1}(T)$. This work will play an important role in later analysis.
	\begin{lemma}\label{boundb1}
		Let Assumption \ref{assu1}, \ref{assu2}-\ref{assu6} hold.
		Suppose that $\varphi(T)$ and $y_{1}(T)$ are defined in (\ref{eqx01}) and (\ref{eq0130}) with the same initial value $\varphi(0)$.
		For $p>1$, there exists a positive constant $C$, such that
		\begin{align}
		& \mathbb{E}\Big(\sup_{0\leq T \leq T_{0}} \|y_{1}(T)\|^{p}\Big)\leq C\|\varphi(0)\|^{p}+C\Big(\mathbb{E}\big(\sup_{0\leq T\leq T_{0}}\|\hat{\mathcal{Z}}(T)\|^{2p}_{\alpha}\big)\Big)^{\frac{1}{2}}.\label{boundb1z}\\
		\intertext{Moreover, if $\|\varphi(0)\|\leq \varepsilon^{-\frac{\kappa}{3}}$,}
		&\mathbb{E}\Big(\sup_{0\leq T\leq \tau^{\star}}\|\varphi(T)-y_{1}(T)\|^{p}\Big)\leq C\varepsilon^{\frac{p}{5}-12\kappa p},\label{eq0921}\\
		&\mathbb{E}\Big(\sup_{0\leq T\leq \tau^{\star}}\|\varphi(T)\|^{p}\Big)\leq C\varepsilon^{-\frac{\kappa p}{3}}.\label{eq095}
		\end{align}
	\end{lemma}
	\begin{proof}
		The proof of this lemma is similar to that of Lemma \ref{boundb}, so we just show the key procedures.
		Define some stopping time
		\begin{equation*}
		\tau_{K_{1}}:=\inf\{T>0, \|y_{1}(T)\|>K_{1}\}.
		\end{equation*}
		
		For $p\geq 2$ and $0 \leq T\leq \tau_{K_{1}}$, It$\hat{\textup{o}}^{,}$s formula yields that
		\begin{equation*}
		\begin{split}
		\|y_{1}(T)\|^{p}&\leq\|\varphi(0)\|^{p}+p\int^{T}_{0}\|y_{1}\|^{p-2}\langle\bar{\mathcal{L}}y_{1},y_{1}\rangle\textup{d}s+2p\int^{T}_{0}\|y_{1}\|^{p-2}\langle\mathcal{F}(y_{1}),y_{1}\rangle\textup{d}s\\
		&\quad+p\int^{T}_{0}\|y_{1}\|^{p-2}\langle y_{1},\Gamma(y_{1},\hat{\mathcal{Z}})\textup{d}\widetilde{W}(s)\rangle\\
		&\quad+Cp(p-1)\int^{T}_{0}(\|y_{1}\|^{p}+\|y_{1}\|^{p-2}\|\hat{\mathcal{Z}}\|^{2})\textup{d}s.
		\end{split}
		\end{equation*}
		For $T_{1}\in[0,T_{0}]$, Utilizing Cauchy-Schwarz inequality, Burkholder-Davis-Gundy inequality and Young's inequality, we derive that
		\begin{equation*}
		\begin{split}
		&\mathbb{E}\Big(\sup_{0\leq T\leq T_{1}\wedge\tau_{K}}\|y_{1}(T)\|^{p}\Big)\\
		\leq&C\|\varphi(0)\|^{p}+C\mathbb{E}\Big(\sup_{0\leq  T\leq T_{0}}\|\hat{\mathcal{Z}}(T)\|^{p}\Big)+C\mathbb{E}\Big(\int^{T_{1}\wedge\tau_{K}}_{0}\|y_{1}(s)\|^{p}\textup{d}s\Big)\\
		&+C\mathbb{E}\Big(\int^{T_{1}\wedge\tau_{K}}_{0}\|y_{1}(s)\|^{2p}\textup{d}s\Big)^{\frac{1}{2}}\\
		\leq &C\|\varphi(0)\|^{p}+C\mathbb{E}\Big(\sup_{0\leq  T\leq T_{0}}\|\hat{\mathcal{Z}}(T)\|^{p}\Big)+C\int^{T_{1}}_{0}\mathbb{E}\Big(\sup_{0\leq s_{1}\leq s\wedge\tau_{K}}\|y_{1}(s_{1})\|^{p}\Big)\textup{d}s\\
		&+\frac{1}{2}\mathbb{E}\Big(\sup_{0\leq T\leq T_{1}\wedge\tau_{K}}\|y_{1}(T)\|^{p}\Big).
		\end{split}
		\end{equation*}
		Gronwall's lemma yields that
		\begin{equation*}
		\mathbb{E}\Big(\sup_{0\leq T\leq T_{0}\wedge\tau_{K}}\|y_{1}(T)\|^{p}\Big)\leq C\|\varphi(0)\|^{p}+C\mathbb{E}\Big(\sup_{0\leq  T\leq T_{0}}\|\hat{\mathcal{Z}}(T)\|^{p}\Big).
		\end{equation*}
		By analogous technique in Lemma \ref{boundb}, we obtain (\ref{boundb1z}).
		
		Introduce $h_{1}(T):=y_{1}(T)-\varphi(T)+R_{\varphi}(T)$.\\
		Note that
		\begin{align*}
		h_{1}(T)&=\int^{T}_{0}\bar{\mathcal{L}}(h_{1}-R_{\varphi})\textup{d}s+2\int^{T}_{0}\mathcal{F}(y_{1})\textup{d}s-2\int^{T}_{0}\mathcal{F}(y_{1}-h_{1}+R_{\varphi})\textup{d}s\\
		&\quad+\int^{T}_{0}\varepsilon^{-1}\bar{G}'_{c}(0)(h_{1}-R_{\varphi})\textup{d}\widetilde{W}+2\int^{T}_{0}B_{c}(R_{\varphi}-h_{1},\mathcal{A}_{s}^{-1}\widetilde{G}\textup{d}\widetilde{W}).
		\end{align*}
		Then, we can obtain (\ref{eq093}) and (\ref{eq095}) by similar procedure in Lemma \ref{boundb}.
		
		We complete the proof.
	\end{proof}
	
	\subsection{Weak convergence for amplitude equations}
	In Subsection \ref{reduced}, we do not eliminate the fast O-U process $\hat{\mathcal{Z}}$ in diffusion term, which causes that the reduced system (\ref{eq0130}) is still affected by $\varepsilon$. In next two Subsections, we attempt to  remove $\hat{\mathcal{Z}}$ from (\ref{eq0130}), obtain amplitude equations, and give rigorous error analysis. This Subsection is devoted to the case that the dimension of kernel space $\mathcal{N}$ is more than one. Briefly, we will provide the amplitude equations after averaging $\hat{\mathcal{Z}}$  in such case, and states that the law of $\varphi(T)$ weakly converge to that of the solution of amplitude equation as $\varepsilon$ tends to $0$. However, the convergence rate we present is not precise as that of Theorem \ref{the1}.
	
	Now let us start with some notations and a key lemma.
	
	Set
	\begin{align*}
	g_{1}(T)&:=\int^{T}_{0}\|\bar{G}'_{c}(0)(y_{1}+\hat{\mathcal{Z}})-2B_{c}(y_{1},\mathcal{A}_{s}^{-1}\widetilde{G} )-B_{c}(I\otimes_{s}\mathcal{A}_{s})^{-1}(\hat{\mathcal{Z}}\otimes_{s}\widetilde{G})\|^{2}_{\mathscr{L}_{2}(U, \mathcal{H}^{\alpha})}\textup{d}s,\\
	g_{2}(T)&:=\int^{T}_{0}\|\bar{G}'_{c}(0)(y_{1})-2B_{c}(y_{1},\mathcal{A}_{s}^{-1}\widetilde{G})\|^{2}_{\mathscr{L}_{2}(U, \mathcal{H}^{\alpha})}\textup{d}s \\
	&~~+\sum^{N}_{k=n+1}\frac{\alpha_{k}^{2}}{2\lambda_{k}}\int^{T}_{0} \|\bar{G}'_{c}(0)e_{k}-B_{c}(I\otimes_{s}\mathcal{A}_{s})^{-1}(e_{k}\otimes_{s}\widetilde{G})\|^{2}_{\mathscr{L}_{2}(U,\mathcal{H}^{\alpha}) }    \textup{d}s, \notag\\
	g_{3}(T)&:=\int^{T}_{0}\|\bar{G}'_{c}(0)(\varphi+\hat{\mathcal{Z}})-2B_{c}(\varphi,\mathcal{A}_{s}^{-1}\widetilde{G} )-B_{c}(I\otimes_{s}\mathcal{A}_{s})^{-1}(\hat{\mathcal{Z}}\otimes_{s}\widetilde{G})\|^{2}_{\mathscr{L}_{2}(U, \mathcal{H}^{\alpha})}\textup{d}s,\notag \\
	g_{4}(T)&:=\int^{T}_{0}\|\bar{G}'_{c}(0)(\varphi)-2B_{c}(\varphi,\mathcal{A}_{s}^{-1}\widetilde{G})\|^{2}_{\mathscr{L}_{2}(U, \mathcal{H}^{\alpha})}\textup{d}s \notag \\
	&~~+\sum^{N}_{k=n+1}\frac{\alpha_{k}^{2}}{2\lambda_{k}}\int^{T}_{0} \|\bar{G}'_{c}(0)e_{k}-B_{c}(I\otimes_{s}\mathcal{A}_{s})^{-1}(e_{k}\otimes_{s}\widetilde{G})\|^{2}_{\mathscr{L}_{2}(U,\mathcal{H}^{\alpha}) }    \textup{d}s.. \notag
	\end{align*}
	\begin{lemma}\label{lemg1}
		Under Assumption \ref{assu1}, \ref{assu2}-\ref{assu6}, we obtain $g_{1}(T)=g_{2}(T)+\bar{\mathcal{O}}(\varepsilon^{\frac{1}{5}-3\kappa})$ and $g_{3}(T)=g_{4}(T)+\bar{\mathcal{O}}(\varepsilon^{\frac{1}{5}-6\kappa})$.
	\end{lemma}
	\begin{proof}
		Note that there exists a positive constant $C$, such that $\|y_{1}\|_{C^{\frac{1}{3}}}\leq C\varepsilon^{-3\kappa}$.
		Recalling the definition of $\mathscr{L}_{2}(U,\mathcal{H}^{\alpha})$ and applying Lemma \ref{l3}, we can prove the first result:
		\begin{align*}
		&~~~~\int^{T}_{0}\|\bar{G}'_{c}(0)(y_{1}+\hat{\mathcal{Z}})-2B_{c}(y_{1},\mathcal{A}_{s}^{-1}\widetilde{G} )-B_{c}(I\otimes_{s}\mathcal{A}_{s})^{-1}(\hat{\mathcal{Z}}\otimes_{s}\widetilde{G})\|^{2}_{\mathscr{L}_{2}(U, \mathcal{H}^{\alpha})}\textup{d}s\\
		&=\sum^{\infty}_{i=1}\int^{T}_{0}\|\bar{G}'_{c}(0)(y_{1}+\hat{\mathcal{Z}})f_{i}-2B_{c}(y_{1},\mathcal{A}_{s}^{-1}\widetilde{G}f_{i} )-B_{c}(I\otimes_{s}\mathcal{A}_{s})^{-1}(\hat{\mathcal{Z}}\otimes_{s}\widetilde{G}f_{i})\|^{2}_{\mathcal{H}^{\alpha}}\textup{d}s\\
		&=\sum^{\infty}_{i=1}\int^{T}_{0}\|\bar{G}'_{c}(0)(y_{1})f_{i}-2B_{c}(y_{1},\mathcal{A}_{s}^{-1}\widetilde{G}f_{i})\|^{2}_{\mathcal{H}^{\alpha}}\textup{d}s\\
		&~~+2\sum^{\infty}_{i=1}\int^{T}_{0}\langle \bar{G}'_{c}(0)(y_{1})f_{i}-2B_{c}(y_{1},\mathcal{A}_{s}^{-1}\widetilde{G}f_{i}),  \bar{G}'_{c}(0)(\hat{\mathcal{Z}})f_{i}      \rangle_{\mathcal{H}^{\alpha}}\textup{d}s\\
		&~~+2\sum^{\infty}_{i=1}\int^{T}_{0}\langle \bar{G}'_{c}(0)(y_{1})f_{i}-2B_{c}(y_{1},\mathcal{A}_{s}^{-1}\widetilde{G}f_{i}),  B_{c}(I\otimes_{s}\mathcal{A}_{s})^{-1}(\hat{\mathcal{Z}}\otimes_{s}\widetilde{G}f_{j})     \rangle_{\mathcal{H}^{\alpha}}\textup{d}s\\
		&~~+\sum^{\infty}_{i=1}\int^{T}_{0}\|\bar{G}'_{c}(0)(\hat{\mathcal{Z}})f_{i}\|^{2}_{\mathcal{H}^{\alpha}}+     \|B_{c}(I\otimes_{s}\mathcal{A}_{s})^{-1}(\hat{\mathcal{Z}}\otimes_{s}\widetilde{G}f_{i})             \|^{2}_{\mathcal{H}^{\alpha}}\textup{d}s\\
		&~~-2\sum^{\infty}_{i=1}\int^{T}_{0}\langle \bar{G}'_{c}(0)(\hat{\mathcal{Z}})f_{i}, B_{c}(I\otimes_{s}\mathcal{A}_{s})^{-1}(\hat{\mathcal{Z}}\otimes_{s}\widetilde{G}f_{i})           \rangle_{\mathcal{H}^{\alpha}}\textup{d}s\\
		&=\int^{T}_{0}\|\bar{G}'_{c}(0)(y_{1})-2B_{c}(y_{1},\mathcal{A}_{s}^{-1}\widetilde{G})\|^{2}_{\mathscr{L}_{2}(U, \mathcal{H}^{\alpha})}\textup{d}s\\
		&~~+\sum^{N}_{k=n+1}\frac{\alpha_{k}^{2}}{2\lambda_{k}}\int^{T}_{0} \|\bar{G}'_{c}(0)e_{k}-B_{c}(I\otimes_{s}\mathcal{A}_{s})^{-1}(e_{k}\otimes_{s}\widetilde{G})\|^{2}_{\mathscr{L}_{2}(U,\mathcal{H}^{\alpha}) }    \textup{d}s+\bar{\mathcal{O}}(\varepsilon^{\frac{1}{5}-3\kappa}).
		\end{align*}
		
		The proof of the second one is similar to the above deduction, so it is omitted.
		
		We complete the proof.
	\end{proof}
	\begin{Remark}
		Although we can obtain better precise error between $g_{1}(T)$ and $g_{2}(T)$ by another averaging technique, $\bar{\mathcal{O}}(\varepsilon^{\frac{1}{5}-6\kappa})$ is enough for us in this paper.
	\end{Remark}
	
	Let operator $\Sigma: \mathcal{N}\rightarrow L(\mathcal{N}, \mathcal{N})$ be defined by
	\begin{align*}\langle\Sigma(\varphi_{1})\varphi_{2}, \varphi_{2}\rangle:&=\langle\sum^{\infty}_{j=1}[\bar{G}'_{c}(0)(\varphi_{1})-2B_{c}(\varphi_{1},\mathcal{A}_{s}^{-1}\widetilde{G})]f_{j},\varphi_{2}\rangle^{2}\\
	&~~+\sum^{N}_{i=n+1}\frac{\alpha_{i}^{2}}{2\lambda_{i}}\langle\sum^{\infty}_{j=1}[\bar{G}'_{c}(0)e_{i}-B_{c}(I\otimes_{s}\mathcal{A}_{s})^{-1}(e_{i}\otimes_{s}\widetilde{G})]f_{j},\varphi_{2}\rangle^{2},
	\end{align*}
	where $\varphi_{1}, \varphi_{2}\in \mathcal{N}$.
	
	According to operator $\Sigma$, we present amplitude equation:
	\begin{equation}\label{eq83}
	\begin{split}
	\textup{d}y&=[\bar{\mathcal{L}}y+2\mathcal{F}(y)]\textup{d}T+\Sigma^{\frac{1}{2}}(y)\textup{d}W_{\mathcal{N}},\\
	y(0)&=a(0),
	\end{split}
	\end{equation}
	where $W_{\mathcal{N}}$ is $\mathcal{N}$-valued Wiener process.
	
	Denote the law of $(\varphi(T), \psi(T))$ stopped at
	$\tau^{\star}_{1}$ by $\mathbb{P}_{\varepsilon}$, and the projection $:C([0,T_{0}],\mathcal{H}^{\alpha})\rightarrow C([0,T_{0}],\mathcal{N})$ by $\pi$.
	\begin{Theorem}\label{theom}
		Let Assumption \ref{assu1}, \ref{assu2}-\ref{assu6} hold.
		If $\|\varphi(0)\|_{\alpha}\leq C$, the sequence of measures $\pi^{\star}\mathbb{P}_{\varepsilon}$ converges weakly to the measure $\mathbb{P}$, the law of $y(T)$.
		
	\end{Theorem}
	\begin{proof}
		The proof of Theorem \ref{theom} is divided into two steps.
		
		{\bf Step 1. Tightness of the sequence of the probability measures $\pi^{\star}\mathbb{P}_{\varepsilon}$.}
		
		Set $T^{\star}:=T\wedge\tau^{\star}$ and introduce $\varphi_{R}(T):=\varphi(T^{\star})-R_{\varphi}(T^{\star})$.\\
		We derive that
		\begin{align*}
		\varphi_{R}(T)&=\varphi_{R}(0)+\int^{T}_{0}\bar{\mathcal{L}}\varphi\textup{d}s+2\int^{T}_{0}\mathcal{F}(\varphi)\textup{d}s+\int^{T}_{0}\Gamma(\varphi, \hat{\mathcal{Z}})\textup{d}\widetilde{W}.
		\end{align*}
		According to It$\hat{\textup{o}}$ Lemma, we obtain that
		\begin{align*}
		\mathbb{E}\|\varphi_{R}(T)\|^{4}&\leq \|\varphi(0)\|^{4}+C\mathbb{E}\big(\int^{T^{\star}}_{0}\|\varphi_{R}(s)\|^{4}+\|R_{\varphi}(s)\|^{4}+\|R_{\varphi}(s)\|^{8}\textup{d}s\big)\\
		&~~+C\mathbb{E}\big(\int^{T^{\star}}_{0}\|\varphi_{R}(s)\|^{2}(\|\varphi_{R}(s)\|^{2}+\|R_{\varphi}(s)\|^{2}+\|\hat{\mathcal{Z}}(s)\|^{2}_{\alpha})\textup{d}s\big)\\
		&\leq \|\varphi(0)\|^{4}+C+\int^{T}_{0}\mathbb{E}\|\varphi_{R}(s)\|^{4}\textup{d}s,
		\end{align*}
		where the last inequality is a consequence of Assumption \ref{assu3}, Lemma \ref{lx} and the fact that $\mathbb{E}(\|\hat{\mathcal{Z}}(T)\|_{\alpha}^{p})$ is uniformly bounded in $[0,T_{0}]$.
		Thus, by Gronwall's Lemma, there exists a positive constant $C_{1}$, such that
		\begin{align}\label{eq101}
		\sup_{0\leq T \leq T_{0}}\mathbb{E}\|\varphi_{R}(T)\|^{4}\leq C_{1}.
		\end{align}
		where $C_{1}$ is a positive constant.
		It follows from (\ref{eq101}) that
		there exists a positive constant $C_{2}$, such that
		\begin{align}\label{eq103}
		\mathbb{E}\|\varphi_{R}(T_{1})-\varphi_{R}(T_{2})\|^{4}\leq C(T_{1}-T_{2})^{2}, \forall T_{1}, T_{2}\in[0,T_{0}].
		\end{align}
		By (\ref{eq82}) and (\ref{eq103}), we conclude 	$\mathbb{E}\|\varphi(T^{\star}_{1})-\varphi(T^{\star}_{2})\|^{4}\leq C_{3}(T_{1}-T_{2})^{2}$, with  a positive constant $C_{3}$.
		Then, we achieve this step by Kolmogorov's criterion for weak compactness \cite{Re}.

		{\bf Step 2. Every accumulation point of $\pi^{\star}\mathbb{P}_{\varepsilon}$ is a solution to the martingale problem associated with (\ref{eq83}).}
		
		Denote smooth and compactly supported function defined in $\mathcal{N}$ by $C^{\infty}_{0}(\mathcal{N})$.
		Considering $\varTheta(\varphi_{R}(T))$ with $\varTheta \in C^{\infty}_{0}(\mathcal{N})$, by It$\hat{\textup{o}}$ formula, we derive
		\begin{equation}\label{eq100}
		\begin{split}
		\varTheta(\varphi_{R}(T))-\varTheta(a(0))&=M^{\varTheta}(T)+\int^{T^{\star}}_{0}\langle D\varTheta(\varphi_{R}),\bar{\mathcal{L}}\varphi+2\mathcal{F}(\varphi)\rangle\textup{d}s\\
		&\quad+\frac{1}{2}\int^{T^{\star}}_{0}\sum^{\infty}_{j=1}\big(D^{2}\varTheta(\varphi_{R})\big)\big(\Gamma(\varphi,\hat{\mathcal{Z}})f{j},\Gamma(\varphi,\hat{\mathcal{Z}})f_{j}\big)\textup{d}s \quad a.s.,
		\end{split}
		\end{equation}
		where $M^{\varTheta}_{\varepsilon}(T)$ is a martingale with respect to $\pi^{\star}\mathbb{P}_{\varepsilon}$.
		
		Due to Lemma \ref{lx}, Lemma \ref{lemg1} and smooth enough function $\varTheta$, we deduce that the third term in the right hand of (\ref{eq100}) converges to
		\begin{equation*}
		\frac{1}{2}\int^{T^{\star}}_{0}\textup{Tr}[\big(D^{2}\varTheta(\varphi)\big)\Sigma(\varphi)]\textup{d}s,
		\end{equation*}
		as $\varepsilon$ tends to $0$.
		Moreover, recalling Lemma \ref{l07} and Lemma \ref{boundb1}, we can easily obtain
		\begin{align*}
		\mathbb{P}(\tau^{\star}=T_{0})\geq 1-\varepsilon^{p}, \text{for}~p>1,
		\end{align*}
		which implies
		\begin{equation}\label{eqtaut}
		\lim_{\varepsilon\rightarrow 0}\mathbb{P}(\tau^{\star}=T_{0})=1.
		\end{equation}
		Then, replacing $T^{\star}$ with $T$, we derive that
		\begin{equation*}
		\begin{split}
		\varTheta(\varphi(T))-\varTheta(\varphi(0))&=\hat{M}^{\varTheta}_{\varepsilon}(T)+\int^{T}_{0}\langle D\varTheta(\varphi),\bar{\mathcal{L}}\varphi+2\mathcal{F}(\varphi)\rangle\textup{d}s\\
		&\quad+\frac{1}{2}\int^{T}_{0}\textup{Tr}[\big(D^{2}\varTheta(\varphi)\big)\Sigma(\varphi)]\textup{d}s+\hat{R}_{\varepsilon}(T)\quad a.s.,
		\end{split}
		\end{equation*}
		where $\hat{M}^{\varTheta}_{\varepsilon}(T)$ is a martingale with respect to $\pi^{\star}\mathbb{P}_{\varepsilon}$ and $\hat{R}_{\varepsilon}(T)$ is an error term.
		Since $\hat{M}^{\varTheta}_{\varepsilon}(T)$ dependent of $\varphi(T)$ stops at $\tau^{\star}$ and $\varTheta$ is smooth enough, we own that
		\begin{equation}\label{eq092}	\lim_{\varepsilon\rightarrow0}{\mathbb{E}_{\varepsilon}\Big(\sup_{0\leq T\leq T_{0}}|\hat{R}_{\varepsilon}(T)|\Big)}=0.
		\end{equation}
		
		Define the continuous function $\bar{M}^{\varTheta}:\mathcal{C}([0,T_{0}],\mathcal{N})\rightarrow \mathcal{C}([0,T_{0}],\mathbb{R})$ by
		\begin{equation*}
		\big(\bar{M}^{\varTheta}(\varphi)\big)(T)=\varTheta(\varphi(T))-\varTheta(\varphi(0))-\int_{0}^{T}
		\big(\mathcal{L}\varTheta\big)(\varphi)\textup{d}s,
		\end{equation*}
		where $\mathcal{L}$ is the infinitesimal generator of (\ref{eq83}).\\
		We proceed to prove that $\bar{M}^{\varTheta}$ is $\mathbb{P}$-martingale where $\mathbb{P}$ is an arbitrary accumulation point of $\pi^{\star}\mathbb{P}_{\varepsilon}$. Let $\pi^{\star}\mathbb{P}_{\varepsilon_{n}}$
		be a sequence of $\pi^{\star}\mathbb{P}_{\varepsilon}$ such that $\pi^{\star}\mathbb{P}_{\varepsilon_{n}}$ weakly converges to $\hat{\mathbb{P}}$
		as $n\rightarrow \infty$.\\
		Then, for any continuous function $f:\mathcal{C}([0,T_{0}],\mathcal{N})\rightarrow\mathbb{R}$ and $T\in[0,T_{0}]$, we obtain
		\begin{equation}\label{eqmt}
		\begin{split}
		&~~\hat{\mathbb{E}}\big(\bar{M}^{\varTheta}(\varphi)\big)(T)f(\varphi)\\
		&=\lim_{n\rightarrow 0}\mathbb{E}_{\varepsilon_{n}}\big(\bar{M}^{\varTheta}(\varphi)\big)(T)f(\varphi)\\
		&=\lim_{n\rightarrow 0}\mathbb{E}_{\varepsilon_{n}}\big(\bar{M}^{\varTheta}(\varphi)-\hat{R}_{\varepsilon_{n}}\big)(T)f(\varphi),
		\end{split}
		\end{equation}
		where the first equality holds since $\bar{M}^{\varTheta}$ and $f$ are continuous functions, and the last equality is evident from (\ref{eq092}).\\
		Note that $\hat{M}^{\varTheta}_{\varepsilon_{n}}(T)=\big(\bar{M}^{\varTheta}(\varphi)-\hat{R}_{\varepsilon_{n}}\big)(T)$ is a martingale with respect to $\pi^{\star}\mathbb{P}_{\varepsilon_{n}}$.\\
		Then, for any $0\leq T_{a}< T_{b}\leq T_{0}$, we get
		\begin{align}\label{eq093}
		\lim_{n\rightarrow 0}\mathbb{E}_{\varepsilon_{n}}\big(\bar{M}^{\varTheta}(\varphi)-\hat{R}_{\varepsilon_{n}}\big)(T_{a})f(\varphi)
		=\lim_{n\rightarrow 0}\mathbb{E}_{\varepsilon_{n}}\big(\bar{M}^{\varTheta}(\varphi)-\hat{R}_{\varepsilon_{n}}\big)(T_{b})f(\varphi),
		\end{align}
		which and (\ref{eqmt}) imply that
		\begin{align*}
		\hat{\mathbb{E}}\big(\bar{M}^{\varTheta}(\varphi)\big)(T_{a})f(\varphi)=\hat{\mathbb{E}}\big(\bar{M}^{\varTheta}(\varphi)\big)(T_{b})f(\varphi).
		\end{align*}
		Thus, $\hat{\mathbb{P}}$ is a solution to the martingale problem associated with (\ref{eq83}). In addition, since there is a unique solution of (\ref{eq83}) due to \cite{St}, we obtain $\hat{\mathbb{P}}=\mathbb{P}$.
		
		We complete the proof.
		
	\end{proof}
	
	Based on Theorem \ref{theom}, we know $\varphi(t)\approx y(t)$ in the law. Then, by Lemma \ref{l02} and Lemma \ref{l30} we claim
	\begin{align*}
	u(t)\approx \varepsilon y(\varepsilon^{2} t)+\varepsilon Q(\varepsilon ^{2}t)+\varepsilon\mathcal{Z}(\varepsilon^{2}t).
	\end{align*}
	
	\subsection{Amplitude equation for one dimensional kernel space}
	In this subsection, our aim is to investigate the case that $\dim \mathcal{N}=1$. Compared with multi-dimensional kernel space, we can apply martingale representation theorem to one construct amplitude equation. With this convenience, we can show explicit convergence rate of the error.
	
	As the amplitude equation is just one dimensional SDE in this subsection, we introduce some notations for convenience of understanding.
	
	Set:
	\begin{align}
	\tilde{y}_{1}(T)&:=\langle y_{1}(T),e_{1}\rangle,  \notag \\
	\sigma_{1}&:=\langle\bar{\mathcal{L}}e_{1},e_{1}\rangle, \sigma_{2}:=2\langle\mathcal{F}(e_{1}),e_{1}\rangle,
	\sigma_{3}=\|\bar{G}'_{c}(0)e_{1}-2B_{c}(e_{1},\mathcal{A}_{s}^{-1}\widetilde{G})\|_{\mathscr{L}_{2}(U,\mathcal{H})}^{2},\notag \\
	\sigma_{4}&:=\sum^{N}_{k=2}\frac{\alpha_{k}^{2}}{2\lambda_{k}} \|\bar{G}'_{c}(0)e_{k}-B_{c}(I\otimes_{s}\mathcal{A}_{s})^{-1}(e_{k}\otimes_{s}\widetilde{G})\|^{2}_{\mathscr{L}_{2}(U,\mathcal{H}) }, \notag \\
	M_{1}(T)&:=\langle\int^{T}_{0}\Gamma(y_{1},\hat{\mathcal{Z}})\textup{d}\tilde{W}(s),e_{1}\rangle.\label{eq28}
	\end{align}
	
	Obviously, $M_{1}$ is a real-valued martingale with the quadratic variation $g_{1}$ given in Lemma \ref{lemg1}. Now we further remove $\hat{\mathcal{Z}}$ from $M_{1}$ by averaging $\hat{\mathcal{Z}}$ in $g_{1}$.
	\begin{lemma}\label{l20}\textup{\cite{Bl1}}
		Let $\bar{M}_{1}(T)$ be a continuous martingale with respect to filtration $(\mathscr{F}_{T})_{T\geq0}$. Denote the quadratic variation of $\bar{M}_{1}(T)$ by $\bar{g}_{1}(T)$ and let $\bar{g}_{2}(T)$ be an arbitrary $\mathscr{F}_{T}$-adapted increasing process with
		$\bar{g}_{2}(0)=0$. Then, there exists a filtration $\tilde{\mathscr{F}}_{T}$ with $\mathscr{F}_{T}\subset\tilde{\mathscr{F}}_{T}$ and a continuous $\tilde{\mathscr{F}}_{T}$ martingale $M_{2}(T)$ with quadratic
		variation $\bar{g}_{2}(T)$ such that, for every $r_{0}<\frac{1}{2}$, there exists a positive constant $C$ with
		\begin{equation*}
		\begin{split}
		\mathbb{E}\sup_{0\leq T\leq T_{0}}\Big|\bar{M}_{1}(T)-\bar{M}_{2}(T)\Big|^{p}&\leq C (\mathbb{E}|\bar{g}_{2}(T_{0})|^{2p})^{\frac{1}{4}}\Big(\mathbb{E}\sup_{0\leq T\leq T_{0}}|\bar{g}_{1}(T)-\bar{g}_{2}(T)|^{p}\Big)^{r_{0}}\\
		&\quad+C\mathbb{E}\sup_{0\leq T\leq T_{0}}|\bar{g}_{1}(T)-
		\bar{g}_{2}(T)|^{\frac{p}{2}}.
		\end{split}
		\end{equation*}
	\end{lemma}
	\begin{lemma}\label{lemb}
		Suppose Assumptions \ref{assu1}-\ref{assu6} hold. Let $M_{1}(T)$ be given in (\ref{eq28}). Then,
		for $p>1$ and $\|\varphi(0)\|\leq \varepsilon^{-\frac{\kappa}{3}}$, there exists a continuous
		$\tilde{\mathscr{F}}_{T}$ martingale $M_{2}(T)$ with the quadratic variation $g_{2}(T)$ and a positive constant $C$, such that
		\begin{equation}\label{eq33}
		\mathbb{E}\Big(\sup_{0\leq T \leq T_{0}}|M_{1}(T)-M_{2}(T)|^{p}\Big)\leq C \varepsilon^{\frac{p}{15}-\frac{4\kappa p}{3}}.
		\end{equation}
		Moreover, there exists a Brownian motion $B(T)$ with respect to the filtration $\tilde{\mathscr{F}}_{T}$, such that
		\begin{equation}\label{eq34}
		M_{2}(T)=\int^{T}_{0}(\sigma_{3}\tilde{y}^{2}_{1}+\sigma_{4})^{\frac{1}{2}}\textup{d}B.
		\end{equation}
	\end{lemma}
	\begin{proof}
		Lemma \ref{lemg1} leads to
		\begin{equation}\label{eq31}
		\begin{split}
		\mathbb{E}(\sup_{0\leq T\leq T_{0}}|g_{2}(T)|^{p})\leq C \mathbb{E}(\sup_{0\leq T\leq T_{0}}\|\tilde{y}_{1}(T)\|^{2p})+C\leq C\varepsilon^{-\frac{2\kappa p}{3}}.
		\end{split}
		\end{equation}
		Choosing $r_{0}=\frac{1}{3}$, it is easy to show (\ref{eq33}) by (\ref{eq31}), Lemma \ref{lem13} and Lemma \ref{l20}.
		In view of martingale representation theorem, we obtain (\ref{eq34}).
	\end{proof}
	
	Our next task is to provide amplitude equation and complete the approximation result. \\
	Firstly, we introduce some equations on stochastic basis $(\Omega, \tilde{\mathcal{F}}_{T},\mathbb{P})$:
	\begin{align}
	\textup{d}\tilde{y}_{1}&=\sigma_{1}\tilde{y}_{1}\textup{d}T
	+\sigma_{2}\tilde{y}_{1}^{3}\textup{d}T+\textup{d}M_{1}(T),~~\tilde{y}_{1}(0)=\tilde{\varphi}(0),\label{eq35}\\
	\textup{d}\tilde{y}_{2}&=\sigma_{1}\tilde{y}_{2}\textup{d}T+\sigma_{2}\tilde{y}_{2}^{3}\textup{d}T+\textup{d}M_{2}(T),~~\tilde{y}_{2}(0)=\tilde{\varphi}(0), \label{eq02}\\
	\textup{d}\tilde{y}_{3}&=\sigma_{1}\tilde{y}_{3}\textup{d}T+\sigma_{2}\tilde{y}_{3}^{3}\textup{d}T+(\sigma_{3}\tilde{y}_{3}^{2}+\sigma_{4})^{\frac{1}{2}}\textup{d}B(T),~~\tilde{y}_{3}(0)=\tilde{\varphi}(0),\label{eq03}
	\end{align}
	where $\tilde{\varphi}(0):=\langle \varphi(0), e_{1}\rangle$ and $B(T)$ is the Brownian motion given in (\ref{eq34}). We note that (\ref{eq03}) is the amplitude equation one desire.
	\begin{lemma}\label{l05}
		Suppose Assumption \ref{assu1}-\ref{assu6} hold. Then, for $p>1$, there exists a positive constant $C$, such that
		\begin{align}
		&\mathbb{E}\Big(\sup_{0\leq T\leq T_{0}}|\tilde{y}_{2}(T)|^{p}\Big)\leq C |\tilde{\varphi}(0)|^{p}+C\Big(\mathbb{E}\big(\sup_{0\leq T\leq T_{0}}\|\hat{\mathcal{Z}}(T)\|^{2p}_{\alpha}\big)\Big)^{\frac{1}{2}}+C.\label{eq05}
		\intertext{Moreover, if $|\tilde{a}(0)|\leq \varepsilon^{-\frac{\kappa}{3}}$,}
		&\mathbb{E}\Big(\sup_{0\leq T\leq T_{0}}|\tilde{y}_{1}(T)-\tilde{y}_{2}(T)|^{p}\Big)\leq C\varepsilon^{\frac{p}{15}-\frac{8\kappa p}{3}}.\label{eq06}
		\end{align}
	\end{lemma}
	\begin{proof}
		Recalling $\sigma_{2}\leq 0$ and $M_{1}-M_{2}=\bar{\mathcal{O}}(\varepsilon^{\frac{1}{15}-\frac{4\kappa}{3}})$, we can get (\ref{eq05}) and (\ref{eq06}) by similar deduction of the proof of Lemma \ref{boundb} and Lemma \ref{boundb1}.
		
		We complete the proof.
	\end{proof}
	\begin{lemma}\label{l06}
		Suppose Assumption \ref{assu1}-\ref{assu6} hold. Then, for $p> 1$, there exists a positive constant $C$, such that
		\begin{align}
		&\mathbb{E}\Big(\sup_{0\leq T\leq T_{0}}|\tilde{y}_{3}(T)|^{p}\Big)\leq C|\tilde{\varphi}(0)|^{p}+C.\label{eq307}
		\intertext{Moreover, if $|\tilde{\varphi}(0)|\leq \varepsilon^{-\frac{\kappa}{3}}$,}
		&\mathbb{E}\Big(\sup_{0\leq T \leq T_{0}}|\tilde{y}_{2}(T)-\tilde{y}_{3}(T)|^{p}\Big)\leq C \varepsilon^{\frac{p}{15}-\frac{8\kappa p}{3}}.\label{eq04}
		\end{align}
	\end{lemma}
	\begin{proof}
		Since the proof of (\ref{eq307}) is similar to that of (\ref{boundb1z}), we do not present the detail.
		
		Let us start to prove (\ref{eq04}).
		
		Introduce a function $g(x)=(\sigma_{3}x^{2}+\sigma_{4})^{\frac{1}{2}}$, and a notation $R_{5}(T)=\tilde{y}_{2}(T)-\tilde{y}_{3}(T)$.\\ Then,
		\begin{equation*}
		R_{5}(T)=\sigma_{1}\int^{T}_{0}R_{5}\textup{d}s+\sigma_{2}\int^{T}_{0}(\tilde{b}^{3}_{2}-\tilde{b}^{3}_{3})\textup{d}s+\int^{T}_{0}g(\tilde{y}_{1})-g(\tilde{y}_{3})\textup{d}B.
		\end{equation*}
		For $p\geq2$, thanks to It$\hat{\textup{o}}^{,}$s formula, we derive
		\begin{align}
		|R_{5}(T)|^{p}&=\sigma_{1}p\int^{T}_{0}|R_{5}|^{p}\textup{d}s+\sigma_{2}p\int^{T}_{0}|R_{5}|^{p-2}R_{5}(\tilde{b}^{3}_{2}-\tilde{b}^{3}_{3})\textup{d}s \notag\\
		 &\quad+p\int^{T}_{0}|R_{5}|^{p-2}R_{3}[g(\tilde{y}_{1})-g(\tilde{y}_{3})]\textup{d}B \notag\\
		&\quad+p(p-1)\int^{T}_{0}|R_{5}|^{p-2}|g(\tilde{y}_{1})-g(\tilde{y}_{3})|^{2}\textup{d}s \notag \\
		&\leq C\int^{T}_{0}|R_{5}|^{p}\textup{d}s+C\int^{T}_{0}|R_{5}|^{p-2}R_{5}[g(\tilde{y}_{1})-g(\tilde{y}_{3})]\textup{d}B \notag\\
		&\quad+C\int^{T}_{0}|R_{5}|^{p}+|\tilde{y}_{1}-\tilde{y}_{2}|^{p}\textup{d}s,\label{eq07}
		\end{align}
		where we use the globally Lipschitz property of $g(x)$.\\
		Since the stochastic integral is a martingale, so we get
		\begin{equation*}
		\begin{split}
		\mathbb{E}\Big(|R_{5}(T)|^{p}\Big)\leq C\int^{T}_{0}\mathbb{E}\Big(|R_{5}(s)|^{p}\Big)\textup{d}s+C\int^{T}_{0}\mathbb{E}\Big(|\tilde{y}_{1}(s)-\tilde{y}_{2}(s)|^{p}\Big)\textup{d}s.
		\end{split}
		\end{equation*}
		By Gronwall's lemma and (\ref{eq06}), we have
		\begin{equation}\label{eq08}
		\mathbb{E}\Big(|R_{5}(T)|^{p}\Big)\leq C\varepsilon^{\frac{p}{15}-\frac{8\kappa p}{3}},~~ T\in[0,T_{0}].
		\end{equation}
		Taking the expectations of the supremum of (\ref{eq07}) on both sides,  by Burkholder-Davis-Gundy inequality, we get
		\begin{equation}\label{eqeq}
		\begin{split}
		\mathbb{E}\Big(\sup_{0\leq T\leq T_{0}}|R_{5}(T)|^{p}\Big)&\leq \mathbb{E}\Big(\int^{T_{0}}_{0}|R_{5}(s)|^{p}\textup{d}s\Big)\\
		&\quad+\mathbb{E}\Big(\int^{T_{0}}_{0}|R_{5}(s)|^{2p}+|\tilde{y}_{1}(s)-\tilde{y}_{2}(s)|^{2p}\textup{d}s\Big)^{\frac{1}{2}}.
		\end{split}
		\end{equation}
		Then, combining (\ref{eq06}), (\ref{eq08}) and (\ref{eqeq}), we obtain (\ref{eq04}).
		
		We complete the proof.
	\end{proof}
	\begin{lemma}\label{l08}
		Let Assumption \ref{assu1}-\ref{assu6} hold. For $p>1$, $\|u\|_{\alpha}\leq -\frac{\kappa}{3}$, there exists a positive constant $C$, such that
		\begin{equation*}
		\mathbb{E}\left(\sup_{0\leq T\leq \tau^{\star}}\|\mathcal{R}_{1}(T)\|^{p}_{\alpha}\right)\leq C\varepsilon^{\frac{p}{15}-12\kappa p},
		\end{equation*}
		where
		\begin{equation*}
		\mathcal{R}_{1}(T)=u(\varepsilon^{-2}T)-\varepsilon \tilde{y}_{3}(T)e_{1}-\varepsilon Q(T)-\varepsilon \mathcal{Z}(T).
		\end{equation*}
	\end{lemma}
	
	\begin{proof}
		Rewriting $\mathcal{R}_{1}$ as follows:
		\begin{equation*}
		\begin{split}
		\mathcal{R}_{1}&=\varepsilon [\varphi(T)+ \psi(T)-\tilde		{y}_{3}(T)e_{1}-Q(T)-\mathcal{Z}(T)]\\
		&=\varepsilon [\varphi(T)-\tilde{y}_{3}(T)e_{1}+J(T)+K(T)]\\
		&=\varepsilon [\varphi(T)-\tilde{y}_{1}(T)e_{1}+\tilde{y}_{1}(T)e_{1}-\tilde{y}_{2}(T)e_{1}+\tilde{y}_{2}(T)e_{1}-\tilde{y}_{3}(T)e_{1}]\\
		&\quad+\varepsilon[J(T)+K(T)].
		\end{split}
		\end{equation*}
		Thanks to Lemma \ref{l02}, \ref{boundb1}, \ref{l05} and \ref{l06}, we can easily prove this lemma by triangle inequality.
	\end{proof}
	\begin{definition}
		Let Assumption \ref{assu1}-\ref{assu6} hold. For $\kappa>0$ , define $\bar{\Omega}^{\star}\subset\Omega$ of all $\omega\subset\Omega$ such that all these estimations
		\begin{equation*}
		\sup_{0\leq T \leq \tau^{\star}}\|\varphi(T)\|< \varepsilon^{-\frac{\kappa}{2}},~~
		\sup_{0\leq T \leq \tau^{\star}}\|\psi(T)\|_{\alpha}<\varepsilon^{-\frac{\kappa}{2}},~~
		\sup_{0\leq T \leq \tau^{\star}}\|\mathcal{R}_{1}(T)\|_{\alpha}< \varepsilon^{\frac{16}{15}-13\kappa}.
		\end{equation*}
		hold.
	\end{definition}
	\begin{lemma}\label{l09}
		For $p>1$, there exists a positive constant $C$, such that
		\begin{equation*}
		\mathbb{P}(\bar{\Omega}^{\star})\geq 1-C\varepsilon^{p}.
		\end{equation*}
	\end{lemma}
	\begin{proof}
		The proof follows from Chebyshev inequality and simple calculation, so it is omitted here.
		
		We complete the proof.
	\end{proof}
	
	Now, we give the main result of this subsection.
	\begin{Theorem}\label{theo1}
		Let Assumption \ref{assu1}-\ref{assu6} hold and $\|u(0)\|_{\alpha}\leq \varepsilon^{1-\frac{\kappa}{3}}$.
		Then, for  $p>1$, there exists a positive constant $C$, such that
		\begin{equation*}
		\mathbb{P}\Big(\sup_{0\leq t\leq\varepsilon^{-2}T_{0}}\|u(t)-\varepsilon \tilde{y}_{3}(\varepsilon^{2}t)e_{1}-\varepsilon Q(\varepsilon^{2}t)-\varepsilon \mathcal{Z}(\varepsilon^{2}t)\|_{\alpha}>\varepsilon^{\frac{16}{15}-13\kappa}\Big)\leq \varepsilon^{p}.
		\end{equation*}
	\end{Theorem}
	
	\begin{proof}
		Note that
		\begin{equation*}
		\bar{\Omega}^{\star}\subseteq\Big\{\omega\Big|\sup_{0\leq T \leq \tau^{\star}_{1}}\|\varphi(T)\|< \varepsilon^{-\kappa},
		\sup_{0\leq T \leq \tau^{\star}_{1}}\|\psi(T)\|_{\alpha}<\varepsilon^{-\kappa}\Big\}\subseteq\{\omega\Big|\tau^{\star}=T_{0}\}\subseteq\Omega.
		\end{equation*}
		Then,
		\begin{equation*}
		\sup_{0\leq T \leq T_{0}}\|\mathcal{R}_{1}(T)\|_{\alpha}=\sup_{0\leq T \leq \tau^{\star}_{1}}\|\mathcal{R}_{1}(T)\|_{\alpha}< \varepsilon^{\frac{4}{3}-13\kappa}, \omega\in \bar{\Omega}^{\star}.
		\end{equation*}
		By Lemma \ref{l09},
		\begin{equation*}
		\mathbb{P}(\sup_{0\leq T \leq T_{0}}\|\mathcal{R}_{1}(T)\|_{\alpha}\geq\varepsilon^{\frac{16}{15}-13\kappa})\leq 1-\mathbb{P}(\bar{\Omega}^{\star})\leq \varepsilon^{p}.
		\end{equation*}
		
		We complete the proof.
	\end{proof}
	
	We would like to give additional remarks before  closing this section.
	\begin{Remark}~\\
		(1). For the case that there is only multiplicative noise in (\ref{eq001}), we can easily obtain the amplitude equation, and prove that the approximation solution converges to the  original one with the  rate $\bar{\mathcal{O}}(\varepsilon^{2-13\kappa})$. \\
		(3). For the case that $\bar{G}_{c}'(0)(\hat{\mathcal{Z}})\cdot-B_{c}(I\otimes \mathcal{A}_{s})^{-1}(\hat{\mathcal{Z}}\otimes\widetilde{G}\cdot)=0$, we just need to deal with the O-U process in drift terms, then obtain the amplitude equation and further prove
		the error between the approximation solution and the  original one is $\bar{\mathcal{O}}(\varepsilon^{\frac{3}{2}-13\kappa})$. \\
		(3). For any $1<h<\frac{5}{4}$, we can choose suitable $\tilde{\alpha}$ in Lemma \ref{lx} and $\gamma_{0}$ in Lemma \ref{lemb} such that
		the convergence rate is Theorem \ref{theo1} to $\bar{\mathcal{O}}(\varepsilon^{h-13\kappa})$.\\
		(4). If the amplitude equation is autonomous, we can analyze the stability of the original system via it. 
	\end{Remark}
	
	\section{Example}
	In the section, we will apply our main results to establish the amplitude equation for the following Burgers' equation with additive and multiplicative noise on the spatial domain $D=[0,\pi]$ subject to  Dirichlet boundary condition:
	\begin{align}\label{eq0105}
	\textup{d}u=[(\partial_{xx}+1)u+\varepsilon^{2}\nu u+u\partial_{x}u]\textup{d}T+(\sigma_{\varepsilon}+\varepsilon u)\textup{d}W(t),
	\end{align}
	where $W(t)$ is introduced later.
	
	Set $\mathcal{H}:=L^{2}[0,\pi]$, $\mathcal{A}:=\partial_{xx}+1$, $\mathcal{L}:=\nu I$ and $B(u,v):=\frac{1}{2}u\partial_{x}v+\frac{1}{2}v\partial_{x}u$
	
	Note that $-\mathcal{A}e_{k}(x)=\lambda_{k}e_{k}(x)$ with $\lambda_{k}=k^{2}-1$ and $e_{k}(x)=\sqrt{\frac{2}{\pi}}\sin{kx}$, where $e_{k}(x)$
	is an orthonormal basis of $\mathcal{H}$. Thus, $\mathcal{A}$ satisfies Assumption \ref{assu1} with $m=2$ and $\mathcal{N}=\{e_{1}\}$. Then Assumption \ref{assu1} is true with $m=2$ and $\mathcal{N}=\{e_{1}\}$. and
	$P_{c}$, the $\mathcal{H}-$ orthogonal projection on $\mathcal{N}$, commutes with $\mathcal{A}$.
	
	Assumption \ref{assu2} holds for the case $\alpha=\frac{1}{4}$, $\beta=\frac{5}{4}$. Note $\mathcal{H}^{\frac{1}{4}}$
	is with base $f_{k}:=k^{-\frac{1}{4}}e_{k}$.
	
	Obviously,
	\begin{equation*}
	P_{c}B(\sin{x}, \sin{x})=0,
	\end{equation*}
	and H\textup{\"{o}}lder inequality and Sobolev embedding theorem yield
	\begin{equation*}
	2\|B(u,v)\|_{\mathcal{H}^{-1}}=\|\partial_{x}(uv)\|_{\mathcal{H}^{-1}}\leq\|uv\|_{L^{2}}\leq C\|u\|_{L^{4}}\|v\|_{L^{4}}\leq C\|u\|_{\mathcal{H}^{\frac{1}{4}}}\|v\|_{\mathcal{H}^{\frac{1}{4}}}.
	\end{equation*}
	Thus Assumption \ref{assu3} is true.
	
	Set  $\mathcal{F}(u,v,w):=-B_{c}(u,\mathcal{A}_{s}^{-1}B_{s}(v,w)),~u,v,w\in\mathcal{N}$.
	Let us check Assumption \ref{assu4}.
	\begin{equation*}
	\mathcal{F}(u_{1}\sin{x}, u_{2}\sin{x}, u_{3}\sin{x})=-\frac{1}{24}u_{1}u_{2}u_{3}\sin(x)
	\end{equation*}
	implies that $\mathcal{F}$ is a trilinear, symmetric and continuous map.
	For $u_{1}, u_{2}, u_{3}\neq 0$, there exists $C_{0}>0$ such that
	\begin{align*}
	\|\mathcal{F}(u_{1}\sin{x}, u_{2}\sin{x}, u_{3}\sin{x})\|\leq C_{0}\|u_{1}\|\|u_{2}\|\|u_{3}\|,\\
	\end{align*}
	Moreover, it is easy to find $C_{1}, C_{2}, C_{3}$ such that (\ref{eq44}) holds, so $\mathcal{F}$ satisfies Assumption \ref{assu4}.
	
	$W(t)$ is standard cylindrical $\mathcal{H}$-valued Wiener process with covariance operator $\mathcal{Q}$ on a stochastic base
	$(\Omega,\mathscr{F},\{\mathscr{F}_{t}\}_{t\geq 0},\mathbb{P})$. Define $\mathcal{Q}$ by $\mathcal{Q}e_{k}=\alpha^{2}_{k}e_{k}$
	with $\alpha_{k}=0$, for $k=4,\cdots$.
	Define $G(u)$ by $G(u)\cdot v:=u\mathcal{Q}^{\frac{1}{2}}v$. Clearly,
	$G(\cdot):\mathcal{H}^{\frac{1}{4}}\rightarrow\mathscr{L}^{2}(\mathcal{H},\mathcal{H}^{\frac{1}{4}})$
	is a Hilbert-Schmidt operator such that all conditions of Assumption \ref{assu6} are satisfied.
	
	\textbf{Case I}:~$\sigma_{\varepsilon}=\varepsilon^{2}$. We consider the amplitude equation of (\ref{eq098}). \\
	Based on (\ref{eq012}), we derive the amplitude equation:
	\begin{equation}\label{eq030}
	\textup{d}\tilde{x}=(\nu \tilde{x}-\frac{1}{12 }\tilde{x}^{3})\textup{d}T+\alpha_{1}\textup{d}\tilde{\beta}_{1}(T)+\frac{8\sqrt{2}\alpha_{1}}{3\pi^{\frac{3}{2}}}\tilde{x}\textup{d}\tilde{\beta}_{1}(T)-\frac{8\sqrt{2}\alpha_{3}}{15\pi^{\frac{3}{2}}}\tilde{x}\textup{d}\tilde{\beta}_{3}(T).
	\end{equation}
	According to Theorem \ref{the1}, we state
	\begin{align*}
	u(t)\approx \varepsilon \tilde{x}(\varepsilon^{2} t)\sin(x).
	\end{align*}
	If $\alpha_{1}=0$, we note that the Stratonovich version of (\ref{eq030}) is
	\begin{equation*}\label{eq031}
	\textup{d}\tilde{x}=[(\nu-\frac{64\alpha_{3}^{2}}{225\pi^{3}}) \tilde{x}-\frac{1}{12}\tilde{x}^{3}]\textup{d}T-\frac{8\sqrt{2}\alpha_{3}}{225\pi^{\frac{3}{2}}}\tilde{x}\circ\textup{d}\tilde{\beta}_{3}(T).
	\end{equation*}
	Then according to \cite{Ma}, the constant solution 0 is locally stable if $\nu<\frac{64\alpha_{3}^{2}}{225\pi^{3}}$ and locally unstable if $\nu<\frac{64\alpha_{3}^{2}}{225\pi^{3}}$. By Theorem \ref{the1}, we conclude that if $\alpha_{3}$ is large enough, the multiplicative noise could stabilize the dynamics of (\ref{eq098}) with high probability.
	
	\textbf{Case II}:~$\sigma_{\varepsilon}=\varepsilon$. We consider the amplitude equation of
	\begin{align}\label{eq0106}
	\textup{d}u=[(\partial_{xx}+1)u+\varepsilon^{2}\nu u+u\partial_{x}u]\textup{d}T+(\varepsilon+\varepsilon u)\textup{d}W(t).
	\end{align}
	Note that Assumption \ref{assu3-1} is satisfied due to $B_{c}(\sin{kx},\sin{kx})=0$, for $k>n$. We further assume $\alpha_{1}=\alpha_{2}=0$.\\ Then, we obtain the amplitude equation  by (\ref{eq03}):
	\begin{align*}
	\textup{d}\tilde{y}=[(\nu-\frac{\alpha_{3}^{2}}{4048\pi})\tilde{y}-\frac{1}{12}\tilde{y}^{3}]\textup{d}T+(\frac{128\alpha_{3}^{2}}{225\pi^{3}}\tilde{y}^{2}+\frac{5184\alpha_{3}^{4}}{1225\pi^{3}})^{\frac{1}{2}}\textup{d}B(T),
	\end{align*}
	where $B(T)$ is a real-valued Brownian motion.
	Furthermore, according to Theorem \ref{theo1}, we state
	\begin{align*}
	u(t)\approx\varepsilon \tilde{y}(\varepsilon^{2}t)\sin{x}+e^{-8t}\frac{2}{\pi}\langle u(0), \sin{3x}\rangle_{\alpha}\sin{3x}+\varepsilon\int^{t}_{0}e^{-8(t-s)}\alpha_{3}\textup{d}\beta_{3}(t)\sin{3x}.
	\end{align*}
	
\end{document}